\documentclass[10pt]{amsart}
\newtheorem{theorem}{Theorem}
\newtheorem{lemma}[theorem]{Lemma}
\newtheorem{corollary}[theorem]{Corollary}
\newtheorem{proposition}[theorem]{Proposition}
\newtheorem{conjecture}[theorem]{Conjecture}
\newtheorem{claim}[theorem]{Claim}
\theoremstyle{definition}
\newtheorem{definition}[theorem]{Definition}
\newtheorem{example}[theorem]{Example}
\newtheorem{remark}[theorem]{Remark}

\newtheorem{open problem}[theorem]{Open Problem}
\numberwithin{equation}{section}

\usepackage[vcentermath, enableskew]{youngtab}
\usepackage{mathtools}

\usepackage{enumerate}
\usepackage{tikz}
\usepackage{wasysym}
\usepackage{amssymb}
\usepackage{multicol}
\usepackage{bm}
\usepackage{ytableau}
\usepackage{varioref}
\newcommand*\circled[1]{\tikz[baseline=(char.base)]{
            \node[shape=circle,draw,inner sep=0.7pt] (char) {#1};}}

\newcommand\sdot{\mathrel{\ooalign{$\subseteq$\cr
  \hidewidth\raise0.275ex\hbox{$\cdot\mkern2.9mu$}\cr}}}
\newcommand{\f}{*\cdots*}

\usepackage{listings}
\usepackage{sourcecodepro} 
\usepackage{xcolor}

\definecolor{codegreen}{rgb}{0,0.6,0}
\definecolor{codegray}{rgb}{0.5,0.5,0.5}
\definecolor{codepurple}{rgb}{0.58,0,0.82}
\definecolor{backcolour}{rgb}{0.95,0.95,0.92}

\lstdefinestyle{mystyle}{
    backgroundcolor=\color{backcolour},   
    commentstyle=\color{codegreen}\let\lst@um\@empty\hspace{-0.1em},
    keywordstyle=\color{blue},
    numberstyle=\tiny\color{codegray},
    stringstyle=\color{codepurple},
    basicstyle=\ttfamily\tiny, 
    xleftmargin=-4em, 
    xrightmargin=-4em, 
    morecomment=[l]{\#}
}

\title{Combinatorics of Double Grothendieck Polynomials}

\author{Graham Hawkes}
\thanks{This work was completed while supported by the Max Planck Institute for Mathematics in Bonn.}
\address{Max Planck Institute for Mathematics, Vivatsgasse 7, 53111 Bonn, Germany}
\begin{document}

\begin{abstract}
We give a proof of the generalized Cauchy identity for double Grothendieck polynomials, a combinatorial interpretation of the stable double Grothendieck polynomials in terms of triples of tableaux, and an interpolation between the stable double Grothendieck polynomial and the weak stable double Grothendieck polynomial.  This so-called half weak stable double Grothendieck polynomial evaluated at $x=y$ generalizes the type $B$ Stanley symmetric function of Billey and Haiman and is $Q$-Schur positive by degree.   We conclude with two open problems as well as a conjecture regarding the $K$-theoretic analogues of factorial Schur $Q$-functions defined by Ikeda and Naruse. The conjecture is supported by code given in the appendices.
\end{abstract}
\maketitle

\section{Introduction}
\subsection{Background}

Grothendieck polynomials \cite{LS82},\cite {LS83},  are a non-homogeneous generalization of Schubert polynomials, the latter a family of polynomials indexed by permutations which are studied among other things in relation to the combinatorics of Coxeter groups.  In particular, the lowest degree term of a Grothendieck polynomial is a Schubert polynomial.  Combinatorially, Grothendieck polynomials replace the notion of the symmetric group with that of the $0$-Hecke monoid \cite{BKSTY08}.  Double Schubert polynomials are  considered in \cite{LS82a} to generalize Schubert polynomials by extending them to two sets of variables in such a way that setting the second variable set to zero returns a regular Schubert polynomial.   In turn,  double Grothendieck polynomials generalize Grothendieck polynomials by extending them to two sets of variables in such a way that setting the second variable set to zero returns a regular Grothendieck polynomial. 
 
Double Grothendieck polynomials themselves are generally not symmetric in either set of variables.  However, there exists a way to derive a (doubly) symmetric function from a double Grothendieck polynomial through a process of taking a stable limit.  These limits are known as the  stable double Grothendieck polynomials.

Stable double Grothendieck polynomials expand in terms of another class of polynomials called the balanced double Grothendieck polynomials.   These polynomials are doubly Schur positive (Lenart (\cite{Lenart00}) proves symmetric (single) Grothendieck polynomials are Schur positive it follows from Fomin-Kirillov (\cite{FK94})  that balanced double Grothendieck polynomials can be written as a sum of  products of the former) and have nice combinatorial interpretations in terms of set valued tableaux \cite{Buch02}, \cite{McNamara06}.   We also mention there exist weak versions of the  stable (double) and symmetric (double)  Grothendieck polynomials.  In particular the weak symmetric  Grothendieck polynomials have a combinatorial interpretation in terms of multiset valued tableaux \cite{LaPy07}.

\subsection{Contributions and Organization} We explain what this paper contributes and how it is structured simultaneously. In section 2 we recall various constructions and results pertaining to Grothendieck polynomials appearing elsewhere in the literature that will be needed for the rest of the paper.  

Section 3 deals with the most general version of double Grothendieck polynomials that we consider in this paper.  The main result of section 3 is Theorem \ref{Cauch} which proves three formulae for the double Grothendieck polynomial.  The first two are combinatorial expressions in terms of certain factorizations of Hecke words.  Such interpretations are more useful for our purposes than pipe dream formulations as they are amenable to the Hecke insertion of \cite{BKSTY08}.  Moreover, the relation between our two models helps to explicate the relationship between double Grothendieck and single Grothendieck polynomials. This relation is made explicit in the proof of the third formula, which is a generalization of the Cauchy identity for Schubert polynomials and which can be credited to \cite{FK94}.  Here we give a proof of this formula. 

Section 4 deals with stable double Grothendieck polynomials, balanced double Grothendieck polynomials and their relation.  The main result of section 4 is Theorem \ref{tabt} which gives a formula for the stable double Grothendieck polynomial in terms of triples of tableaux (see remark \ref{tabl}) as well as a similar formula for the weak double Grothendieck polynomial.  If it were not for the fact that  Hecke insertion lacks a certain property (see remark \ref{pita}) such an expression would be an easy corollary of the work of \cite{BKSTY08}.  Indeed in almost all imaginable analogous cases (i.e., for choices of the parameters single vs double, type $A$ vs other types, standard vs $K$-theoretic)  that have been defined a similar expression follows directly from the relevant insertion algorithm.  However, in the absence of the property of Remark \ref{pita} of Hecke insertion (we leave it as an open problem to amend Hecke insertion so that it \emph{does} have this property)  some additional work must be done.  This work comprises the majority of section 4.  

In section 5 we investigate the relation between stable double Grothendieck polynomials (of type $A$) and a potential definition of  stable (single) Grothendieck polynomials of type $B/C$.  The latter definition is given by first taking an interpolation of the stable double Grothendieck polynomial and the weak stable double Grothendieck polynomial and then evaluating the result when the two sets of variables are set equal to each other (i.e., at $x=y$).  As we will see, the combinatorial definition of this  ``half weak" double Grothendieck polynomial is more amenable to Hecke insertion than either the weak or non-weak versions.  Moreover, unlike the others, it is $Q$-Schur positive at $x=y$.
This is the main result of section 5, which is stated as Theorem \ref{QP} which also gives a combinatorial interpretation of the coefficients in the $Q$-Schur expansion.  The half weak double Grothendieck polynomial (evaluated at $x=y$) generalizes nicely the type $B$ Stanley symmetric function of \cite{Billey.Haiman.1995}: In particular, the lowest degree part of the former function recovers the latter.  In addition, it is also $Q$-Schur positive by degree.

Section 6 includes two open problems and a conjecture that arise from our study of double Grothendieck polynomials.

\section{Single Grothendieck Polynomials}
\subsection{Operator definition}
\begin{definition}
Let $f \in \mathbb{Z}[x_1, \ldots x_{n+1}]$.  For each $1 \leq i \leq n$ define the divided difference operator $\partial_i$ by $\partial_i(f)=\frac{f-s_i f}{x_i-x_{i+1}}$ where $s_i$ acts by interchanging the variables $x_i$ and $x_{i+1}$.  Define $\pi_i$ by the formula $\pi_i(f)=\partial_i(f)+\partial_i(x_{i+1}f)$.
\end{definition}
\begin{lemma}\label{relations}
The divided difference operators satisfy the following relations: \cite{LS82}, \cite{LS83}
\begin{enumerate}
\item If $|i-j|>1$ then $\partial_i\partial_j=\partial_j\partial_i$ and $\pi_i\pi_j=\pi_j\pi_i$. 
\item If $i=j+1$ then $\partial_i\partial_j\partial_i=\partial_j\partial_i\partial_j$ and $\pi_i\pi_j\pi_i=\pi_j\pi_i\pi_j$.
\item $\partial_i^2=0$.
\item $\pi_i^2=-\pi_i$.
\end{enumerate}
\end{lemma}
Given a permutation $w \in S_n$ one can write down (non-uniquely in general) $w$ as  sequence of adjacent transpositions, i.e.: $w=s_{i_1}\cdots s_{i_{\ell}}$ where $\ell$ is the inversion number of the permutation.  We can then define $\partial_{w}$ by $\partial_{i_1} \cdots \partial_{i_{\ell}}$ and  $\pi_{w}$ to be $\pi_{i_1} \cdots \pi_{i_{\ell}}$.  By parts 1 and 2 of Lemma \ref{relations} this procedure is well defined, i.e., the definition of $\partial_{w}$ and $\pi_{w}$  does not depend on the chosen reduced word.  
\begin{definition}
Fix $w \in S_n$ and let $w_0$ refer to the element of $S_n$ with maximal inversion number.  
Define the Grothendieck polynomial for $w$ by:
\begin{eqnarray*}
 \mathfrak{G}_{w}(x)=\pi_{(w^{-1}w_0)}(x_1^nx_2^{n-1} \cdots x_n^1x_{n+1}^0).
\end{eqnarray*}
\end{definition}
See \cite{LS82} and \cite{LS83} for original formulations.
\subsection{Hecke Insertion}
 Consider a new  operator $\bar{s_i}$ (which we will frequently just write as $i$)  acting on permutations of the set $\{\mathbf{1},\mathbf{2},\mathbf{3},\ldots,\mathbf{n},\mathbf{n+1}\}$ where the operation $\bar{s_i}$ or $i$ is given by interchanging $\mathbf{i}$ and $\mathbf{i+1}$ if $\mathbf{i}$ lies to the left of $\mathbf{i+1}$ and by doing nothing otherwise (In particular $i^2=i$ whereas $s_i^2=e$).  In this setting, a  \emph{Hecke word} for $w$ is a sequence $i_1,\ldots,i_k$ such that  applying this sequence (right to left) to the starting arrangement $\{\mathbf{1},\mathbf{2},\mathbf{3},\ldots,\mathbf{n},\mathbf{n+1}\}$ gives the permutation $w$.  
 
We will give an overview of a simple insertion algorithm \cite{BKSTY08} for Hecke words which will be necessary at various stages.  First we need to define two  types of tableaux:
\begin{definition}
A standard set-valued tableau of shape $\lambda$ is a filling of a Young diagram of shape $\lambda$ with exactly one of each of  the letters $\{1,\ldots,N\}$ for some integer $N \geq |\lambda|$ such that each box contains at least one entry and such that all entries in a given box are smaller than all the entries in the box below and smaller than all the entries in the box to the right.
\end{definition}
\begin{definition}
Given a permutation $w$, a Hecke tableau for $w$ of shape $\lambda$ or element of $HT_{w}(\lambda)$ is a tableau where each box of $\lambda$ is filled with exactly one of the symbols $\{1,\ldots,n\}$ in such a way that reading the boxes by rows, moving left to right within the rows and  moving bottom to top amongst the rows gives a Hecke word for $w$, and, such that the rows and columns are strictly increasing in the order $1< \cdots <n$. 

\end{definition}

To define Hecke insertion\footnote{In this paper we consider Hecke insertion as a row insertion algortihm (e.g., \cite{PP16})  rather than a column insertion algorithm.}, we first show how to insert some $a  \in \{1,\ldots, n\}$ into some row of a Hecke tableau, say $V=(v_1,v_2,\ldots, v_j)$ read from left to right.  Suppose that the row above $V$ (if it exists) is $U$ and it has entries $U=(u_1,u_2,\ldots,u_k)$ and the row below it is $Z$ with entries $Z=(z_1,z_2,\ldots,z_i)$.  Here all $u,v,z  \in \{1,\ldots,n\}$.  We assume that $a  \in [u_h,v_h)$ for some $h$ (where possibly $h>j$ and $v_h$ is taken, by convention to be $\infty$) and that $a>u_1$ if $U$ exists.  There are no restrictions on $a$ if $U$ does not exist that is, if $V$ is the first row of the tableau.

\Yboxdim{20pt}\young({{u_1}}{{u_2}}\cdots{{u_i}}\cdots{{u_j}}{{  u_ {j+1}    }}\cdots{{u_k}},{{v_1}}{{v_2}}\cdots{{v_i}}\cdots{{v_j}},{{z_1}}{{z_2}}\cdots{{z_i}})  
\begin{tabular}{ccc}
$\longleftarrow a$ \\
 \\
 \\
\\
\end{tabular}

We insert $a$ into $V$ as follows:
\begin{enumerate}
\item If $a \geq v_j$ and:
\begin{itemize}
\item $a>v_j$ and $a>u_{j+1}$. Then $a$ is appended to the right of $v_j$.  
\item $a=v_j$ or $a=u_{j+1}$. Then $a$ simply disappears.
\end{itemize}
\item If $a<v_j$. Let $h$ be minimal such that $a \leq v_h$.
\begin{itemize}
\item $a=v_h$.  Then $V$ stays the same and $v_{h+1}$ is inserted into $Z$.
\item $a<v_h$ and $a>u_h$  then $a$ replaces $v_h$ and $v_h$ is inserted into row $Z$.
\item $a<v_h$ and $a=u_h$  then $V$ is unchanged and $v_h$ is inserted into row $Z$.
\end{itemize}
\end{enumerate}
Note that the result is strictly decreasing down columns by construction and that our assumption on $a$ guarantees one of the situations above must occur.  Moreover, the assumption is maintained moving on to the next insertion.  That is, the element (if it exists) to be inserted into row $Z$ exists in some interval $[v_{\ell},z_{\ell})$ and is greater than $v_1$.

We now describe complete Hecke insertion. Given a Hecke word say $\alpha_1  \cdots  \alpha_m$ create a sequence of pairs of tableaux of the same shapes $(P_0,Q_0),(P_1,Q_1),\ldots, (P_m,Q_m)$ by setting $P_0=\emptyset=Q_0$ and creating $(P_{i+1},Q_{i+1})$ from $(P_i,Q_i)$ as follows.
Insert $\alpha_{i+1}$ into $P_i$ by inserting it into the first row of $P_i$.  As long as there is an output,  insert the output into the next row.  The algorithm stops when either an element is appended to the end of a row or disappears. The resulting Hecke tableau is $P_{i+1}$.  If the algorithm ends by appending an element, add a box to the corresponding position of  $Q_i$ and fill it with the number $i+1$ to form $Q_{i+1}$. If the algorithm stops by an element disappearing, take the row where the last insertion occurred and caused this element to disappear and consider its rightmost box $b$.  Now find the lowest box in the same column as $b$, call it $b'$.  Add an $i+1$ to the position corresponding to $b'$  in $Q_i$ to form $Q_{i+1}$.  (Of course, it is possible $b'=b$.)

$\young(\hfil\hfil\hfil\hfil\hfil\hfil,\hfil\hfil\hfil\hfil\hfil\hfil,\hfil\hfil\hfil \hfil\hfil,\hfil\hfil\hfil{{b}},\hfil\hfil\hfil\hfil,\hfil\hfil\hfil{{b'}},\hfil\hfil\hfil)$ $\longleftarrow$ disappearing element\\
\begin{example}
Suppose that we have
\begin{eqnarray*} 
P_{10}=\young(245,468,57,7,8) \,\,\,\,\,\,\,\,\,\,\,\,\,\,\, Q_{10}=\young(124,367,5{{10}},8,9) \,\,\,\,\,\,\,\,\,\,\,\,\,\,\, \alpha_{11}=3
\end{eqnarray*}

Then $P_{11}$ and $Q_{11}$ are computed as follows.  

\begin{enumerate}
\item First the $3$ is inserted into row one of $P_{11}$.  $3$ replaces $4$ in this row and $4$ is sent to be inserted into row two. 
\item  $4$ is inserted into row two. Since a $4$ already appears in row two this row does not change and the number to the right of the $4$ in row two, which is $6$, will be inserted into row three.  
\item When $6$ is inserted into row three it would replace the $7$ with itself except that the number above this $7$ in row two is not less than $6$ (it is $6$).  Thus row three remains unchanged and the $7$ is inserted into row four.  
\item Row four ends in $7$ itself so the inserted $7$ is disappeared.  
\end{enumerate}

To form the recording tableau an $11$ is added not in the box, $b$,  at the end of row four but to the box $b'$ at the bottom of the column containing $b$.  All in all the only changes are in the first row of the insertion tableau and the fifth row of the recording tableau and the result is:
\begin{eqnarray*} 
P_{11}=\young(2{{\textcolor{red}{3}}}5,468,57,7,8) \,\,\,\,\,\,\,\,\,\,\,\,\,\,\, Q_{11}=\young(124,367,5{{10}},8,{{9,\textcolor{red}{11}}}) 
\end{eqnarray*}
\end{example}

\begin{proposition}[\cite{BKSTY08}] \label{Hebij}
Fix $ w \in S_{n+1}$.  Hecke insertion is a bijection from Hecke words for $w$ to pairs $(P,Q)$ of tableaux of the same shape where $P$ is a Hecke tableau for $w$ and $Q$ is a standard set-valued tableau.
\end{proposition}
Additionally, Hecke insertion also has the following convenient property:

\begin{lemma}[\cite{BKSTY08}] \label{combined}
 If the word $\alpha_1\cdots \alpha_m$ maps to $(P,Q)$. Then $\alpha_i>\alpha_{i+1}$ if and only if $i+1$ shows up in a row strictly below $i$ in $Q$.
\end{lemma}

\begin{remark}\label{pita}
Unfortunately, (this fact will cause mild consternation later) the following statement \textbf{is not true} of Hecke insertion: \emph{Suppose  $\alpha_1\cdots \alpha_m$ maps to $(P,Q)$. Then $\alpha_i < \alpha_{i+1}$ if and only if $i+1$ shows up in a column to the right of  $i$ in $Q$.}  In fact, neither the \emph{if} nor the \emph{only if} part of this statement is  true.  For instance applying Hecke insertion to $1322$ results in:
\begin{eqnarray*}
\left\{\young(1),\young(1)\right\}, \left\{\young(13),\young(12)\right\}, \left\{\young(12,3),\young(12,3)\right\},\left\{\young(12,3),\young(1{{24}},3)\right\}
\end{eqnarray*}
showing that although $4$ shows up in a column to the right of $3$ in $Q$, it is not true that $\alpha_3=2<2=\alpha_4$.
Conversely, applying Hecke insertion to $1312$ results in:
\begin{eqnarray*}
\left\{\young(1),\young(1)\right\}, \left\{\young(13),\young(12)\right\}, \left\{\young(13,3),\young(12,3)\right\},\left\{\young(12,3),\young(12,{{34}})\right\}
\end{eqnarray*}
showing that although  $\alpha_3=1<2=\alpha_4$, it is not true that $4$ shows up in a column to the right of $3$ in $Q$.
\end{remark}

\begin{definition}
A Hecke word that has been partitioned into groups of transpositions with decreasing indices is known as a \emph{Hecke factorization}. For instance  $(32)(321)()(1)$ is a  Hecke factorization with four factors for the permutation $(4,1,3,2) \in S_4$.  If $\mathfrak{f}$ is a Hecke factorization we denote by $wt(\mathfrak{f})$ the vector whose $i^{th}$ coordinate records the number of entries in the $i^{th}$ factor (from left to right) of $\mathfrak{f}$.  In general, we will be concerned with Hecke factorizations that have a fixed number of factors.  This number of factors will usually be clear from context, but whenever that number is not of particular importance we will assume it is equal to $m+1$ where $m$ is a positive integer that should be assumed to be the same throughout its appearances in different definitions.
\end{definition}

\begin{definition}
A Hecke factorization of an element in $S_{n+1}$ with $n+1$ factors that has only entries with indices of at least $i$ in the $i^{th}$ subdivision is known as a \emph{bounded  Hecke factorization}. For instance, $(321)(2)(3)()$ is a bounded Hecke factorization in $S_4$.  (Note that the definition implies the last factor is always empty.)   
\end{definition}

\begin{itemize}
\item Let $\mathcal{F}_{w}$ denote the set of all (unbounded)  Hecke factorizations of $w$   into $m+1$ parts for some integer $m$.
\item Let $\mathfrak{F}_{w}$ denote the set of all bounded Hecke factorizations of $w$ into $n+1$ parts.
\end{itemize}

\begin{definition}
A \emph{conjugate (semistandard) set-valued tableau} or $CSVT$ of shape $\lambda$ is a filling of a Young diagram of shape $\lambda$ with the  letters $\{1,\ldots,m,m+1\}$ (repetition allowed) for some integer $m$ such that each box contains a nonempty set of numbers and
\begin{itemize}
\item If box $b$ lies to the left of box $b'$ then $max(b) < max (b')$.
\item If box $b$ lies above box $b'$ then $max(b) \leq max(b')$.
\end{itemize}
A skew $CSVT$ of shape $\lambda/\mu$ is defined similarly.
\end{definition}

\begin{definition}
 \emph{Semistandard Hecke insertion} is the following algorithm.  Starting with $\mathfrak{f} \in \mathcal{F}_{w}$ first consider the underlying Hecke word, $\alpha$, given by erasing the parentheses in $\mathfrak{f}$.  Apply regular Hecke insertion to obtain a pair of tableaux $(P,Q)$.  Now form $Q'$ from $Q$ as follows.  Wherever $j$ appears in $Q$ replace the $j$ with an $i$ where $i$ is chosen such that the $j^{th}$ entry in $\mathfrak{f}$ appears in the $i^{th}$ factor of $\mathfrak{f}$.  The result of the algorithm is the pair $(P,Q')$.
\end{definition}

We end this section with a few results that will be useful later.

\begin{proposition}\cite{BKSTY08}\label{SSbij}
 Semistandard Hecke insertion is a weight preserving bijection  from $\mathcal{F}_{w}$ to pairs $(P,Q)$ where $P \in HT_{w}$ and $Q \in CSVT$ have the same shape.
\end{proposition}

\begin{theorem}\cite{Lascoux90}\label{sam}  We have that  
\begin{eqnarray*}
 \mathfrak{G}_{w}({x})=\sum_{\mathfrak{f} \in {\mathfrak{F}}_{w}  }   {x}^{wt(\mathfrak{f})}.
\end{eqnarray*}
\end{theorem}

\begin{lemma} \cite{Buch02}
Suppose that $\ell_{1}\geq \ell_{2}$.  We have: \label{tabtopi}
\begin{eqnarray*}
\sum_{Q \in CSVT(\ell_1, \ell_2)} x_r^{(\#1)}x_{r+1}^{(\#2)}=\pi_{s_r}(x_r^{\ell_1+1}x_{r+1}^{\ell_2})
\end{eqnarray*}
where the left hand sum is over all $CSVT$ with two columns, of lengths $\ell_1$ and $\ell_2$, in the letters $\{1,2\}$ and where $(\#1)$ and $(\#2)$ is the number of $1$s and $2$s respectively in $Q$.
\end{lemma}

\section{Double Grothendieck Polynomials}
Consider two sets of variables $x=(x_1,\ldots,x_{n+1})$ and $y=(y_1,\ldots,y_{n+1})$.  We extend the action of $\pi_{s_r}$ linearly over $\mathbb{Z}[y]$ to get an action on $\mathbb{Z}[x,y]$.   

\begin{definition}
A \emph{circled Hecke factorization} is a factorization of a Hecke word into factors, where some of the elements have been circled.  Moreover, each factor must be decreasing in the order $\circled{1}<1<\circled{2}< \cdots <\circled{n} <n$. For instance  $\big(32\circled{2}\big)\big(\circled{3}21\circled{1}\big)\big(\big)\big(\circled{1}\big)$ is a circled Hecke factorization for the permutation $(4,1,3,2) \in S_4$.  
\end{definition}

\begin{definition}
 A \emph{bounded circled  Hecke factorization} is a circled  Hecke factorization with $n+1$ factors such that all the elements in the $i^{th}$ factor are $\geq \circled{i}$. For instance, $\big(\circled{4}3\circled{2}1\big)\big(3\circled{3}\big)\big(43\circled{3}\big)\big(4\big)\big(\big)$ is a bounded circled Hecke factorization for the permutation $(5,1,4,3,2) \in S_5$.  The \emph{$x$-weight} of  a bounded circled  Hecke factorization is the vector whose $i^{th}$ entry records the number of uncircled elements in its $i^{th}$ factor. The $x$-weight of the example above is $(2,1,2,1,0)$.   The \emph{$y$-weight} of such a factorization is the vector whose $i^{th}$ entry records the number of circled entries that have some value $j$ and appear in some factor $k$ such that $j-k+1=i$.  The $y$-weight of the example above is $(1,2,0,1,0)$. 
 \end{definition}

\begin{definition}
A \emph{double Hecke factorization} is a factorization into an even number of factors where the first half of the factors are increasing in the order $1 < \cdots <n$ and the last half of the factors are decreasing in the order $1 < \cdots <n$.  For example $(12)(13)|(21)(32)$ is a double Hecke factorization for $(4,3,2,1) \in S_4$, where we have drawn a ``$|$" between the left half and the right half of the factors for viewing convenience.  (Such factorizations are combinatorially equivalent to (non-reduced)  pipe dreams as used to express double Grothendieck polynomials.  See \cite{M16} for example.)
\end{definition}

\begin{definition}
A \emph{bounded double Hecke factorization} is a double Hecke factorization into $2n+2$ factors where all elements in the $i^{th}$ factor to the right of center are $\geq i$ and all elements in the $i^{th}$ factor to the left of center are $\geq i$.    For instance, $()(3)(2)(12)|(31)(32)(3)()$ is a bounded double  Hecke factorization for the permutation $(4,3,2,1) \in S_4$. The \emph{$x$-weight} of  a bounded or unbounded double  Hecke factorization is the vector whose $i^{th}$ entry records the number of  elements in the $i^{th}$ factor to the right of center. The $x$-weight of the example above is $(2,2,1,0)$.   The \emph{$y$-weight} of  a bounded or unbounded double Hecke factorization is the vector whose $i^{th}$ entry records the number of  entries in the $i^{th}$ factor to the left of center.  The $y$-weight of the example above is $(2,1,1,0)$.
\end{definition} 

We will use the following notation:

\begin{itemize}
\item Let ${\mathcal{F}}^{\ocircle}_{w}$ denote the set of all (unbounded) circled Hecke factorizations of $w$   into $m+1$ parts. 
\item Let ${\mathfrak{F}}^{\ocircle}_{w}$ denote the set of all bounded circled Hecke factorizations of $w$ into $n+1$ parts.
\item Let ${\mathcal{F}}^{\square}_{w}$ denote the set of all (unbounded) double Hecke factorizations of $w$ into $2m+2$ parts.
\item Let ${\mathfrak{F}}^{\square}_{w}$ denote the set of all bounded double Hecke factorizations of $w$ into $2n+2$ parts.
\end{itemize}

 If $\mathfrak{f}$ is one of the factorizations above we write $(x,y)^{wt(\mathfrak{f})}$to mean the monomial $x^{{wt_x}(\mathfrak{f})}y^{{wt_y}(\mathfrak{f})}$ where $wt_x(\mathfrak{f})$ and $wt_y(\mathfrak{f})$ refer to the $x$-weight and $y$-weight of $\mathfrak{f}$ respectively. If two Hecke words represent the same permutation, we denote this by writing a ``$\sim$" between them.  Moreover, if $\mu$ is any permutation, let $\widetilde{\mu}$ denote an arbitrary Hecke word for $\mu$. Finally, if $\mu$ is any  permutation let $X_{\mu}$ be the set of all pairs of permutations $(u,v)$ such that the concatenation $\widetilde{u}\widetilde{v}$ represents the permutation $\mu$.   Finally, we need one more definition before we can state the main result:

\begin{definition}\label{norm}  The \emph{double Grothendieck polynomial} for $w$ is \cite{Lascoux.85}:
\begin{eqnarray*}
 \mathfrak{G}_{w}(x,y)=\pi_{(w^{-1}w_0)}\left(\prod_{i+j \leq n+1} x_i+y_j +x_iy_j\right).
\end{eqnarray*}
\end{definition}

The rest of this section will be devoted to proving that:
\begin{theorem}\label{Cauch}
We have: 
\begin{align}
 \mathfrak{G}_{w}(x,y)=\sum_{\mathfrak{f}\in {\mathfrak{F}}^{\ocircle}_{w}} (x,y)^{wt(\mathfrak{f})}& \\
 \mathfrak{G}_{w}(x,y)=\sum_{\mathfrak{f}\in {\mathfrak{F}}^{\square}_{w}} (x,y)^{wt(\mathfrak{f})}& \text{\cite{FK94}, \cite{KM04}}\\
 \mathfrak{G}_{w}(x,y)=\sum_{(u,v) \in X_{w}} \mathfrak{G}_{u^{-1}}(y) \mathfrak{G}_v(x) &\text{\cite{FK94},\cite{McNamara06},\cite{BFHTW}}.
\end{align}
\end{theorem}
The theorem will follow immediately from combining the three main Lemmas of this section: \ref{stb},  \ref{cts}, and  \ref{ntb}.

\begin{lemma}\label{stb}
We have \begin{eqnarray*}
\sum_{\mathfrak{f}\in {\mathfrak{F}}^{\square}_{w}} (x,y)^{wt(\mathfrak{f})}=\sum_{(u,v) \in X_{w}} \mathfrak{G}_{u^{-1}}(y)  \mathfrak{G}_v(x).
\end{eqnarray*}
\end{lemma}
\begin{proof}
We have:
\begin{eqnarray*}
\sum_{\mathfrak{f}\in {\mathfrak{F}}^{\square}_{w}} (x,y)^{wt(\mathfrak{f})}=\sum_{(u,v) \in X_{w}}  \sum_{\mathfrak{f}\in {\mathfrak{F}}^{\square}_{(u,v)}} (x,y)^{wt(\mathfrak{f})}
\end{eqnarray*}
where ${ \mathfrak{F}}^{\square}_{(u,v)}$ is the subset of  ${\mathfrak{F}}^{\square}_{w}$ such that the left  (resp. right) $n+1$ factors give a Hecke word for the permutation $u$ (resp. $v$).  Since the left side of a bounded double Hecke factorization for $u$ is just a bounded Hecke factorization for $u^{-1}$ (written in reverse order) the lemma follows.
\end{proof}
 If $\mu$ is a permutation then define
\begin{itemize}
\item $\vee(\mu)=\{(a,b): ab \sim \widetilde{\mu}$,  $a$ is strictly decreasing, $b$ is strictly increasing.$\}$
\item $\wedge(\mu)=\{(a,b): ab \sim \widetilde{\mu}$,  $a$ is strictly increasing, $b$ is strictly decreasing.$\}$
\end{itemize}
\begin{claim}
There is a bijection from $\wedge(\mu) \rightarrow \vee (\mu)$  such that, denoting word length by $|\cdot|$, if $(b,c) \rightarrow (a,d)$ then $|a|=|c|$  and $|b|=|d|$.
\end{claim}
\begin{proof}
Denote by $W^k(\mu)$ the set of all quadruples of Hecke words $(a,b,c,d)$ such that the concatenation $abcd$ is a Hecke word for $\mu$ and such that
\begin{itemize}
\item $a$ and $c$ are strictly decreasing.
\item $b$ and $d$ are strictly increasing.
\item $b$ and $c$ only contain elements from the set $\{1,2,\ldots,k\}$.
\item $a$ and $d$ only contain elements from the set $\{k+1,\ldots,n\}$.
\end{itemize}
It suffices to find a bijection $W^{k+1}(\mu) \rightarrow W^k(\mu)$, such that if $(a,b,c,d) \rightarrow (a',b',c',d')$ then $|a|+|c|=|a'|+|c'|$ and $|b|+|d|=|b'|+|d'|$.  The bijection is given by the identity in the case that no Hecke word for $\mu$ contains a $k+1$.  If some Hecke word for $\mu$ does contain a $k+1$ then set $K=k+1$.  In this case the bijection fixes all entries which are not equal to $k$ or $K$ and changes the entries equal to $k$ or $K$ as follows:

\begin{align*}
(\f)(\f kK)(Kk \f)(\f) &\rightarrow 
(\f K)(\f k)(k \f)(K\f)\\
(\f)(\f kK)(k \f)(\f)  &\rightarrow 
(\f K)(\f k)( \f)(K\f)\\
(\f)(\f k)(Kk \f)(\f)  &\rightarrow 
(\f K)(\f )(k \f)(K \f)\\
(\f)(\f kK)(K \f)(\f) &\rightarrow 
(\f )(\f k)(k \f)(K\f)\\
(\f)(\f K)(Kk \f)(\f) &\rightarrow 
(\f K)(\f k)(k \f)(\f)\\
(\f)(\f K)(K \f)(\f)  &\rightarrow 
(\f K)(\f )( \f)(K \f)\\
(\f)(\f K)(k \f)(\f)  &\rightarrow 
(\f K)(\f k)( \f)( \f)\\
(\f)(\f k)(K \f)(\f)  &\rightarrow 
(\f )(\f )(k \f)(K \f)\\
(\f)(\f kK)( \f)(\f)  &\rightarrow 
(\f )(\f k)( \f)(K \f)\\
(\f)(\f )(Kk \f)(\f)  &\rightarrow 
(\f K)(\f )(k \f)( \f)\\
(\f)(\f K)( \f)(\f)  &\rightarrow 
(\f )(\f )( \f)(K \f)\\
(\f)(\f )(K \f)(\f)  &\rightarrow 
(\f K )(\f )( \f)( \f)\\
\end{align*}
It is easy to see that this defines a bijection with the desired properties.
\end{proof}
We denote the map $\wedge(\mu) \rightarrow \vee(\mu)$ by $\downarrow$ and its inverse by $\uparrow$.
\begin{example}
Let $(123568)(8752) \in \wedge (\mu)$.  Thus:
\begin{eqnarray*}
W^8=()(123568)(8752)()\\
W^7=(8)(123567)(752)()\\
W^6=(8)(12356)(652)(7)\\
W^5=(86)(1235)(52)(67)\\
W^4=(865)(123)(2)(567)\\
W^3=(865)(123)(2)(567)\\
W^2=(8653)(12)()(3567)\\
W^1=(8653)(1)()(23567)\\
W^0=(8653)()()(123567)
\end{eqnarray*}
so that $\downarrow(123568)(8752)=(8653)(123567)  \in {\vee}(\mu)$.
\end{example}
\begin{lemma}\label{cts}
We have \begin{eqnarray*}
\sum_{\mathfrak{f}\in {\mathfrak{F}}^{\ocircle}_{w}} (x,y)^{wt(\mathfrak{f})}=\sum_{\mathfrak{f}\in {\mathfrak{F}}^{\square}_{w}} (x,y)^{wt(\mathfrak{f})}.
\end{eqnarray*}
\end{lemma}
\begin{proof}
We need to find a bijection from ${\mathfrak{F}}^{\ocircle}_{w}$ to ${\mathfrak{F}}^{\square}_{w}$ that preserves the $x$-weight and the $y$-weight. 
The arguments are quite technical and the reader is encouraged to use example \ref{techex} as a running example while reading through the proof.

 For each $k \in \{1,2,\ldots,n\}$ and $j \in  \{0,1,\ldots,n-k+1\}$ we define the set ${\mathfrak{F}}^{jk}_{w}$ 
to be the set of factorizations $\mathfrak{f}=f_{-(n+1)},\ldots,f_{-(k+1)},f_1,\ldots f_j,f_{mov},f_{j+1},\ldots,f_{n+1}$ such that the Hecke word given by this factorization represents $w$ and where:

\begin{itemize}
\item The factors  $f_{-(n+1)},f_{-(n)},\ldots,f_{-(k+1)}$ are called the \emph{left} factors and:
\begin{enumerate}
\item Each left factor $f_{-(i)}$ only contains elements from the set $\{i,\ldots,n\}$.
\item Each left factor is strictly increasing in the order $1<\cdots<n$. 
\end{enumerate}
\item The factors $f_1,\ldots,f_j,f_{j+1},\ldots ,f_{n+1}$ are called the \emph{right} factors and:
\begin{enumerate}
\item The uncircled elements that each right factor $f_{i}$ contains must come from the set $\{i,\ldots,n\}$.
\item For $i \leq j$, the circled elements that the right factor $f_{i}$ contains must come from the set $\{\circled{i},\ldots, \circled{s}\}$ where $s=k+i-1$.
\item For $i > j$, the circled elements that the right factor $f_{i}$ contains must come from the set $\{\circled{i},\ldots, \circled{t}\}$ where $t=k+i-2$.
\item The right  factors strictly decrease in the order $\circled{1}<1<\cdots<\circled{n}<n$.
\end{enumerate}
\item The factor  $f_{mov}$, or the \emph{moving} factor, has the properties:
\begin{enumerate}
\item $f_{mov}$ contains elements from $\{j+k,\ldots,n\}$.
\item  $f_{mov}$ is strictly increasing in the order $1<\cdots<n$.
\end{enumerate}
\end{itemize}

If $\mathfrak{f} \in {\mathfrak{F}}^{jk}_{w}$ we define $(x,y)^{wt(\mathfrak{f})}$ to be the monomial such that the power of $x_i$ is the number of uncircled elements in $f_i$.  For $i>k$ the power of $y_i$ is the number of elements in $f_{-i}$.  For $i<k$ the power of $y_i$ is the number of times some $\circled{m}$ appears in some factor $f_{\ell}$ such that $m-\ell+1=i$.    The power of $y_k$ is the number of times some $\circled{m}$ appears in some factor $f_{\ell}$ such that $m-\ell+1=k$ plus the number of elements in $f_{mov}$.   
(Essentially what we want to do now is show that we can take a factorization of the form $f_{-(n+1)}\cdots f_{-(k+1)}f_1\cdots f_jf_{mov}f_{j+1}\cdots f_{n+1}$  and  move the moving factor $f_{mov}$  from the right of $f_j$ to its left via some process.  If we repeat this process eventually we can pull $f_{mov}$ all the way to the left of $f_1\cdots f_{n+1}$ and make it $f_{-k}$. If this in turn can be done for each $k$ it means that if we start with a factorization of the form $f_1 \cdots f_{n+1}  \in {\mathfrak{F}}^{\ocircle}_{w}$ we can get one of the form $f_{-(n+1)} \cdots f_{-1} f_1 \cdots f_{n+1}$. We will then want to show the latter lies in  $\mathfrak{F}^{\square}_{w}$ .)

\begin{example}\label{techex}
Let $w=(4,3,2,1) \in S_4$.  The following sequence of factorizations would be computed under the bijection described later in this proof to get from an element of ${\mathfrak{F}}^{\ocircle}_{w}$ to an element of ${\mathfrak{F}}^{\square}_{w}$.  The moving factor is shown in red.  Note that for fixed $k$ the moving factor always moves to the left.  When $k$ decreases by one, the new moving factor starts one position to the right of where the moving factor for the previous $k$ started.
\begin{align*}
&&\big(3\circled{3}\circled{2}1\circled{1}\big)
\big(\circled{3}2\big)
\big(3\circled{3}\big)
 \big(\big)\in \mathfrak{F}_{w}^{\ocircle}\\
\textcolor{blue}{()}&|&
\big(3\circled{3}\circled{2}1\circled{1}\big)
\textcolor{red}{()}
\big(\circled{3}2\big)
\big(3\circled{3}\big)
 \big(\big)\in \mathfrak{F}_{w}^{13}\\
 \textcolor{blue}{()}&|&
 \textcolor{red}{(3)}
\big(3\circled{2}1\circled{1}\big)
\big(\circled{3}2\big)
\big(3\circled{3}\big)
 \big(\big)\in \mathfrak{F}_{w}^{03}\\
  \textcolor{blue}{()(3)}&|&
\big(3\circled{2}1\circled{1}\big)
\big(\circled{3}2\big)
 \textcolor{red}{()}
\big(3\circled{3}\big)
\big(\big) \in \mathfrak{F}_{w}^{22}\\
  \textcolor{blue}{()(3)}&|&
 \big(3\circled{2}1\circled{1}\big)
  \textcolor{red}{(3)}
\big(2\big)
\big(3\circled{3}\big)
\big(\big) \in \mathfrak{F}_{w}^{12}\\
   \textcolor{blue}{()(3)}&|&
 \textcolor{red}{(23)}
 \big(21\circled{1}\big)
\big(2\big)
\big(3\circled{3}\big)
\big(\big) \in \mathfrak{F}_{w}^{02}\\
    \textcolor{blue}{()(3)(23)}&|&
 \big(21\circled{1}\big)
\big(2\big)
\big(3\circled{3}\big)
 \textcolor{red}{()}
\big(\big) \in \mathfrak{F}_{w}^{31}\\
     \textcolor{blue}{()(3)(23)}&|&
 \big(321\circled{1}\big)
\big(2\big)
 \textcolor{red}{(3)}
\big(3\big)
\big(\big) \in \mathfrak{F}_{w}^{21}\\
      \textcolor{blue}{()(3)(23)}&|&
 \big(321\circled{1}\big)
 \textcolor{red}{(2)}
\big(3\big)
\big(3\big)
\big(\big) \in \mathfrak{F}_{w}^{11}\\
       \textcolor{blue}{()(3)(23)}&|&
        \textcolor{red}{(13)}
 \big(321\big)
\big(3\big)
\big(3\big)
\big(\big) \in \mathfrak{F}_{w}^{01}\\
        \textcolor{blue}{()(3)(23)(13)}&|&
 \big(321\big)
\big(3\big)
\big(3\big)
\big(\big) \in \mathfrak{F}_{w}^{\square}\\
\end{align*}
\end{example}

We begin by noting that ${\mathfrak{F}}^{1n}_{w}={\mathfrak{F}}^{\ocircle}_{w}$ and that  ${\mathfrak{F}}^{01}_{w}={\mathfrak{F}}^{\square}_{w}$.  Moreover, we have that ${\mathfrak{F}}^{0k}_{w}={\mathfrak{F}}^{(n-k+2)(k-1)}_{w}$ for $k>1$.  Hence it suffices to find an $x$-weight and $y$-weight preserving bijection from ${\mathfrak{F}}^{jk}_{w}$ to ${\mathfrak{F}}^{(j-1)k}_{w}$ for $j \in \{1,\ldots,n-k+1\}$ and $k \in \{1,\ldots,n\}$. 
To do the latter it suffices to find a bijection, $\Psi_{jk}$ between pairs $(f_j,f_{mov})$ such that:
\begin{itemize}
\item $f_{j}$ contains uncircled elements only from $\{j,\ldots,n\}$.
\item $f_{j}$ contains circled elements only from $\{\circled{j},\ldots, \circled{s}\}$ where $s=j+k-1$.
\item $f_{mov}$ contains elements from $\{j+k,\ldots,n\}$
\item  $f_{mov}$ is strictly increasing in the order $1<\cdots<n$.
\item $f_j$ is strictly decreasing in the order $\circled{1}<1<\cdots<\circled{n}<n$.
\end{itemize}
to pairs $(f_{mov}',f_{j}')$ such that:
\begin{itemize}
\item $f_{j}'$ contains uncircled elements only from $\{j,\ldots,n\}$.
\item $f_{j}'$ contains circled elements only from $\{\circled{j},\ldots, \circled{t}\}$ where $t=j+k-2$.
\item $f_{mov}'$ contains elements from $\{j+k-1,\ldots,n\}$
\item  $f_{mov}'$ is strictly increasing in the order $1<\cdots<n$.
\item $f_j'$ is strictly decreasing in the order $\circled{1}<1<\cdots<\circled{n}<n$.
\end{itemize}
with the property that if $(f_j,f_{mov}) \rightarrow (f_{mov}',f_{j}')$ then the Hecke words $f_jf_{mov}$ and $f_{mov}'f_j'$ represent the same permutation and whenever $(f_j)(f_{mov})$ appears in an element of ${\mathfrak{F}}^{jk}_{w}$ (in the expected position) these two factors make the same contribution to the $x$-weight and $y$-weight as the pair $(f_{mov}')(f_{j}')$ when it appears in an element of ${\mathfrak{F}}^{(j-1)k}_{w}$ (in the expected position).

To do this write $f_j=f_j^>f_j^=f_j^<$ where $f_j^>$, $f_j^=$, $f_j^<$ are  the parts of $f_j$ composed, respectively, of elements greater than, equal to, or less than $\circled{s}$ (where $s=j+k-1$) in the order $\circled{1}<1<\cdots<\circled{n}<n$.  If $f_{j}^=$ is nonempty then append $s$ to the left of $f_{mov}$ to form $f_{mov}^+$.  Otherwise set $f_{mov}^+ =f_{mov}$.  Then let $(g_1,g_2)=\uparrow(f_j^>,f_{mov}^+)$.  We define $\Psi_{jk}(f_j,f_{mov})=(f_{mov}',f_j')$ where $f_{mov}'=g_1$ and $f_j'=g_2f_j^<$.

On the other hand given a pair $(f'_{mov},f_j')$ write $f_j'=f_j'^{>}f_j'^{<}$ where $f_j'^{>}$ and $f_j'^{<}$ are  the parts of $f_j'$ composed, respectively, of elements greater than or less than $\circled{s}$ in the order $\circled{1}<1<\cdots<\circled{n}<n$.  Next set $(h_1,h_2)=\downarrow(f_{mov}',f_j'^>)$.  Now write $h_2=h_2^=h_2^>$  where $h_2^=$ and $h_2^>$ are  the parts of $h_2$ composed respectively of elements equal to or greater than $s$ in the order $1<\cdots<n$.  We define $\Psi_{jk}^{-1}(f_{mov}',f_{j}')=(f_{j},f_{mov})$ where $f_{j}=h_1f_j'^<$ if $h_2^=$ is empty and $f_{j}=h_1\circled{s}f_j'^<$ otherwise and $f_{mov}=h_2^>$.

The fact that $\uparrow$ and $\downarrow$ preserve the permutation represented along with the commutation of nonadjacent transpositions implies that $\Psi$ and $\Psi^{-1}$ do not change the permutation represented.  Moreover the constructions of $\Psi$ and $\Psi^{-1}$ make it clear that they map into the proper images.  One can easily check that $\Psi^{-1} \circ \Psi$ is the identity by considering the two cases where either $f_j$ contains a $\circled{s}$ or does not.   Similarly, one can  check that $\Psi \circ \Psi^{-1}$ is the identity by considering the two cases where either the $h_2$ of the construction of $\Psi^{-1}$ contains an $s$ or does not. 

 Finally, if $(f_j,f_{mov}) \rightarrow (f_{mov}',f_j')$, we need to check these two pairs make the same contributions to the $x$-weight and $y$-weight of the factorization they are part of.  For the first pair, the contribution to the $x$-weight is simply to add $r$ to the $j^{th}$ coordinate of the $x$-weight where $r$ is the number of uncircled elements in $f_j$.   For the second pair, the contribution to the $x$-weight is simply to add $r'$ to the $j^{th}$ coordinate of the $x$-weight where $r'$ is the number of uncircled elements in $f_j'$.    Clearly the construction of $\Psi$ implies that $r=r'$.  
Now the circled elements of $f_j$ and $f_j'$ other than $\circled{s}$ (which only affects the $k^{th}$ coordinate of the $y$-weight since $s-(j-1)=k$) are the same and $f_{mov}$ and $f_{mov}'$ only affect the $k^{th}$ coordinate of the $y$-weight.  Thus it suffices just to check that $(f_j)(f_{mov})$ and $ (f_{mov}')(f_j')$ make the same contribution to the $k^{th}$ coordinate of the $y$-weight.  If $f_j$ does not contain a $\circled{s}$ then $f_{mov}$ and $f_{mov}'$ have the same length $\ell$ and the contribution to the $k^{th}$ coordinate of the $y$-weight is just $\ell+0$ in either case since neither $f_j$ nor $f_j'$ contain a $\circled{s}$.  If $f_j$ does contain a $\circled{s}$ then if $f_{mov}$ has length $\ell$ then  $f_{mov}'$ has length $\ell+1$. The contribution to the $k^{th}$ coordinate of the $y$-weight from $(f_j)(f_{mov})$ is $(1)+(\ell)$ since $f_j$ contains one $\circled{s}$. The contribution to the $k^{th}$ coordinate of the $y$-weight from $ (f_{mov}')(f_j')$ is $(\ell+1)+(0)$ since $f_j'$ contains no $\circled{s}$.
\end{proof}
\begin{example}
Set $n=9$ and $k=3$ and $j=2$.  Suppose that $f_j=(9764\circled{4}\circled{3}2\circled{2})$ and $f_{mov}=(5689)$.  Then as in the construction of $\Psi_{jk}$ we set $s=j+k-1=4$ and $f_j^>$ becomes $(9764)$ while $f_j^=$ becomes $(\circled{4})$ and $f_j^<$ becomes $(\circled{3}2\circled{2})$ .  To compute $f_{mov}^+$ we append a $4$ to $f_{mov}$,  and so $f_{mov}^+$ becomes $(45689)$.  Next we set $(g_1,g_2)=\uparrow(9764)(45689)$. To evaluate this we compute:
\begin{eqnarray*}
W^3=(9764)()()(45689)\\
W^4=(976)(4)(4)(5689)\\
W^5=(976)(45)(5)(689)\\
W^6=(97)(456)(65)(89)\\
W^7=(9)(457)(765)(89)\\
W^8=(9)(4578)(865)(9)\\
W^9=()(45789)(9865)()
\end{eqnarray*}
and see that $g_1=(45789)$ and $g_2=(9865)$.  Therefore we get that $f_{mov}'=g_1=(45789)$ and  $f_j'=g_2f_j^<=(9865\circled{3}2\circled{2})$.  All in all, we see that $\Psi_{23}$ sends
\begin{eqnarray*}
 (9764\circled{4}\circled{3}2\circled{2})(5689)\rightarrow (45789)(9865\circled{3}2\circled{2}).
\end{eqnarray*}
\end{example}
\begin{lemma}\label{ntb}
We have 
\begin{eqnarray*}
 \mathfrak{G}_{w}(x,y)=\sum_{(u,v) \in X_{w}} \mathfrak{G}_{u^{-1}}(y) \mathfrak{G}_v(x).
\end{eqnarray*}
\end{lemma}
\begin{proof}
We proceed by induction on the number of inversions of $w^{-1} w_0$.  First suppose that this number is $0$.  That is, $w=w_0$.  
First we compute
\begin{eqnarray*}\sum_{\mathfrak{f}\in {\mathfrak{F}}^{\circ}_{w_0}} (x,y)^{wt(\mathfrak{f})}.
\end{eqnarray*}
  For any $\mathfrak{f} \in \mathfrak{F}_{w_0}^{\circ}$ the boundedness condition implies that for each $i$, the $i^{th}$ factor of $\mathfrak{f}$ contains a subset of $\{\circled{$i$},i,\ldots,\circled{$n$},n\}$.  Letting $\ell_i$ denote the number of distinct numerical values that appear (uncircled, circled, or both) in the $i^{th}$ factor of $\mathfrak{f}$, it is clear that the inversion number of the permutation represented by $\mathfrak{f}$ is  bounded by $\sum \ell_i$, which, in turn is bounded by $n+(n-1)+\cdots+1+0={n+1 \choose 2}$.  But ${n+1 \choose 2}$ actually is the inversion number of $w_0 \in S_{n+1}$.  Thus $\sum \ell_i={n+1 \choose 2}$, which means that $\ell_i=n+1-i$ for each $i$ or that all the numerical values $\{i,i+1,\ldots,n\}$ show up in the $i^{th}$ factor of $\mathfrak{f}$ (uncircled, circled, or both).  This means that there are precisely $3^{{n+1 \choose 2}}$ factorizations in $\mathfrak{F}_{w_0}^{\circ}$.  Each factorization, $\mathfrak{f}$, is specified by choosing, for each $i \in [1,n+1]$ and $j \in [1,n+1-i]$ whether the value $(i+j-1)$ appears in the $i^{th}$ factor as circled, uncircled, or both.  The value of $(x,y)^{wt(\mathfrak{f})}$ is computed by starting with $1$ and, for each $i \in [1,n+1]$ and $j \in [1,n+1-i]$ multiplying by $x_i$, $y_j$, or $x_iy_j$ depending on whether the value $(i+j-1)$ appears in the $i^{th}$ factor as circled, uncircled, or both. 

It follows that:
\begin{eqnarray*}
\sum_{(u,v) \in X_{w}} \mathfrak{G}_{u^{-1}}(y) \mathfrak{G}_v(x)=\sum_{\mathfrak{f}\in {\mathfrak{F}}^{\circ}_{w_0}} (x,y)^{wt(\mathfrak{f})}=
\prod_{i+j \leq n+1} (x_i+y_j +x_iy_j)=\mathfrak{G}_{w_0}(x,y).
\end{eqnarray*}
where the first equality comes from combining Lemmas \ref{cts} and \ref{stb}. This completes the base step of induction.  
Now suppose that $w^{-1} w_0$ has at least one inversion.  Choose $s_r$ such that $w s_r$ has more inversions than $w$.  
By definition we have $\pi_{s_r} \mathfrak{G}_{w s_r}(x,y)=\mathfrak{G}_{w}(x,y)$ and so if we can show that
\begin{eqnarray} \label{WTS}
\pi_{s_r} \sum_{(u,v) \in X_{w s_r}} \mathfrak{G}_{u^{-1}}(y)\mathfrak{G}_v(x)= \sum_{(u,v) \in X_{w}} \mathfrak{G}_{u^{-1}}(y)\mathfrak{G}_v(x),
\end{eqnarray}
then since the  inductive hypothesis implies  that the left hand sides of the two former equations are equal it will imply that the right hand sides are also equal.  
First we write $X_{w s_r}=A \cup B \cup C$ where:
\begin{itemize}
\item $A =\{(u,v) \in X_{w s_r}$: no reduced word for $v$ ends in $s_r\}$.
\item $B =\{(u,v) \in X_{w s_r}$: $\exists \nu$ s.t. $\widetilde{v} \sim \widetilde{\nu} \bar{s}_r$, $\widetilde{\nu} \not\sim \widetilde{\nu} \bar{s}_r$, $\widetilde{u} \widetilde{\nu} \sim \widetilde{u} \widetilde{\nu} \bar{s}_r\}$.
\item $C =\{(u,v) \in X_{w s_r}$: $\exists \nu$ s.t. $\widetilde{v} \sim \widetilde{\nu} \bar{s}_r$, $\widetilde{\nu} \not\sim \widetilde{\nu} \bar{s}_r$, $\widetilde{u}\widetilde{\nu} \not\sim \widetilde{u}\widetilde{\nu} \bar{s}_r\}$.
\end{itemize}
(Recall that $\bar{s}_r$ is the idempotent or Hecke version of the simple transposition $s_r$.)

\begin{claim}
There exist bijections
\begin{enumerate}[(i)]
 \item $B  \longrightarrow A$ 
\item $C   \longrightarrow X_{w }  $
\end{enumerate}
\end{claim}

\begin{proof}\renewcommand{\qedsymbol}{$\blacksquare$}
We describe a single procedure and its inverse that actually gives both the bijections.  
First let $(u,v) \in   B$ or $(u,v) \in   C$.  In either case we may write $v=\nu s_r$ for a (unique) permutation $\nu$ with $\widetilde{\nu} \not\sim \widetilde{\nu} \bar{s}_r$.  We now simply send $(u,v)$ to $(u,\nu)$.
  
If $(u,v) \in   B$ then $\widetilde{u} \widetilde{\nu} \sim \widetilde{u}\widetilde{\nu} \bar{s}_r$ so that $(u,\nu) \in X_{w s_r}$.  Moreover,  $\widetilde{\nu} \not\sim \widetilde{\nu} \bar{s}_r$ implies that $\nu$ has no reduced word ending in $s_r$, so $(u,\nu) \in A$.  If $(u,v) \in   C$ then $\widetilde{u} \widetilde{\nu} \not\sim \widetilde{u} \widetilde{\nu} \bar{s}_r$  implies that $(u,\nu) \in X_{w }$.

Next let $(u,v) \in   A$  or $(u,v) \in   X_{w }$.  Then the inverse map sends $(u,v)$ to $(u,vs_r)$.   First, suppose that $(u,v) \in   A$. Then the permutation that the Hecke word $\widetilde{u}\widetilde{v}$ represents is $w s_r$ which obviously has a reduced word ending in $s_r$.  Thus the permutation represented by $\widetilde{u}\widetilde{v}\bar{s}_r$ is also $w s_r$.   Further, in the definition of $ B$, replacing $v$ with $vs_r$ and $\nu$ with $v$ satisfies the three requirements of the definition, $(\widetilde{v}\bar{s}_r) \sim \widetilde{v} \bar{s}_r$ (clearly), $\widetilde{v} \not\sim \widetilde{v} \bar{s}_r$ (because $v$ has no reduced word ending in $s_r$), and $\widetilde{u}\widetilde{v} \sim \widetilde{u}\widetilde{v} \bar{s}_r$ (by the previous sentence).  Thus $(u,vs_r) \in   B$.  On the other hand if $(u,v) \in   X_{w}$ then the fact that $\widetilde{u}\widetilde{v}$ represents the permutation $w$ implies that $\widetilde{u}\widetilde{v}\bar{s}_r$ represents the permutation $w s_r$.   Moreover, in the definition of $ C$, replacing $v$ with $vs_r$ and $\nu$ with $v$ satisfies the three conditions:   $(\widetilde{v}\bar{s}_r) \sim \widetilde{v} \bar{s}_r$ (clearly), $\widetilde{v} \not\sim \widetilde{v} \bar{s}_r$ (because the fact $uv$ has no reduced word ending in $s_r$ implies the same for $v$), and $\widetilde{u}\widetilde{v} \not\sim \widetilde{u}\widetilde{v} \bar{s}_r$ (because $\widetilde{u}\widetilde{v} \sim \widetilde{w}$ and $w$ has no reduced word ending in $s_r$).
Finally, it is clear the given maps are mutual inverses.
\end{proof}

\begin{claim}
\begin{align}
\pi_{s_r} \sum_{(u,v) \in B} \mathfrak{G}_{u^{-1}}(y)\mathfrak{G}_v(x)&&=&& -\pi_{s_r} \sum_{(u,v) \in A} \mathfrak{G}_{u^{-1}}(y)\mathfrak{G}_v(x) \label{one} \\
\pi_{s_r} \sum_{(u,v) \in C} \mathfrak{G}_{u^{-1}}(y)\mathfrak{G}_v(x)&&=&&  \sum_{(u,v) \in X_{w}} \mathfrak{G}_{u^{-1}}(y)\mathfrak{G}_v(x) \label{two}
\end{align}
\end{claim}

\begin{proof}\renewcommand{\qedsymbol}{$\blacksquare$}
To prove \ref{one} it suffices to show if $(u, \nu s_r) \rightarrow (u, \nu)$ under bijection $(i)$ then,
\begin{eqnarray*}
\pi_{s_r} \mathfrak{G}_{\nu s_r}(x)= -\pi_{s_r}\mathfrak{G}_{\nu}(x).
\end{eqnarray*}
To prove \ref{two} it suffices to show if $(u, \nu s_r) \rightarrow (u, \nu)$ under bijection $(ii)$ then,
\begin{eqnarray*}
\pi_{s_r} \mathfrak{G}_{\nu s_r}(x)= \mathfrak{G}_{\nu}(x).
\end{eqnarray*}
In either case if $\mu=\mu_1 \cdots \mu_{\ell}$ is a reduced word for the permutation $(\nu s_r)^{-1}w_0$, then $s_r \mu_1 \cdots \mu_{\ell}$ is a reduced word for $\nu^{-1} w_0$.  So by the divided difference operator definition, the equations become:
\begin{eqnarray*}
\pi_{s_r} (\pi_{\mu_1}\cdots \pi_{\mu_{\ell}})(x_1^{n}\cdots x_{n+1}^0)= -\pi_{s_r} (\pi_{s_r} \pi_{\mu_1}\cdots \pi_{\mu_{\ell}})(x_1^{n}\cdots x_{n+1}^0)
\end{eqnarray*}
\begin{center}
and
\end{center}
\begin{eqnarray*}
\pi_{s_r} (\pi_{\mu_1}\cdots \pi_{\mu_{\ell}})(x_1^{n}\cdots x_{n+1}^0)=  (\pi_{s_r} \pi_{\mu_1}\cdots \pi_{\mu_{\ell}})(x_1^{n}\cdots x_{n+1}^0).
\end{eqnarray*}
The first follows by part (4) of Lemma \ref{relations} and the second is immediate.
\end{proof}

Combining equations \ref{one} and \ref{two} gives equation \ref{WTS} thereby completing the induction step and finishing the proof.
\end{proof}

\section{Stable Grothendieck Polynomials}

In this section we discuss certain limits of Grothendieck polynomials as well as the doubled versions of these limits.  The main result of this section is  a formula for the double stable Grothendieck polynomials in terms of set-valued tableaux.    

We will use the following notation in this section: Let $w \in S_{k+1}$ and choose some $m \geq 0$. $m$ will remain fixed throughout this section.  Let $\overrightarrow{w} \in S_{m+k+1}$ be the permutation of $[1,\ldots,m,(m+1), \ldots, (m+k+1)]$ that fixes the first $m$ entries and applies the permutation $w$ to the last $k+1$ entries.  In other words, $s_{i_1}\cdots s_{i_{\ell}}$ is a reduced word for $w$ if and only if $s_{i_1+m}\cdots s_{i_{\ell}+m}$ is a reduced word for $\overrightarrow{w}$.  Let $\overleftarrow{x}=(x_1, \ldots, x_{m+1})$ and $\overrightarrow{x}=(x_{m+2}, \ldots, x_{m+k+1})$.  Similarly let $\overleftarrow{y}=(y_1, \ldots, y_{m+1})$ and $\overrightarrow{y}=(y_{m+2}, \ldots, y_ {m+k+1})$. Write $x=(\overleftarrow{x},\overrightarrow{x})$ and  $y=(\overleftarrow{y},\overrightarrow{y})$.

\subsection{Single Stable and Single Symmetric Grothendieck polynomials}
We quickly review the situation for single Grothendieck polynomials:

\begin{definition}[\cite{Lascoux90}]
The stable Grothendieck polynomial for $w$ is 
\begin{eqnarray*}
{\mathcal{G}}_{w}(\overleftarrow{x})={\mathfrak{G}}_{\overrightarrow{w}}(x)|_{\overrightarrow{x}=0}.
\end{eqnarray*}
\end{definition}

\begin{lemma}[\cite{Lascoux90}]\label{s}
We have
\begin{eqnarray*}
{\mathcal{G}}_{w}(\overleftarrow{x})=\sum_{\mathfrak{f}\in {\mathcal{F}}_{w}} (\overleftarrow{x})^{wt(\mathfrak{f})}.
\end{eqnarray*}
\end{lemma}

\begin{definition}[\cite{Buch02}]
 The conjugate symmetric Grothendieck polynomial for a partition $\lambda$ is 
\begin{eqnarray*}
 {{G}}_{\lambda}(\overleftarrow{x})=\sum_{T \in CSVT(\lambda)} (\overleftarrow{x})^{wt(T)}.
\end{eqnarray*}
The conjugate symmetric Grothendieck polynomial is defined for skew shapes similarly.  
\end{definition}

\begin{proposition}[\cite{BKSTY08}]\label{x}
We have
\begin{eqnarray*}
 {\mathcal{G}}_{w}(\overleftarrow{x})=
\sum_{\lambda} \sum_{T \in HT_{w}(\lambda)} G_{\lambda}(\overleftarrow{x}).
\end{eqnarray*}
\begin{center}
\end{center}
\end{proposition}
\begin{proof}
This follows immediately from Lemma \ref{SSbij}.
\end{proof}

\subsection{Stable Double Grothendieck Polynomials}

\begin{definition}[\cite{Lascoux90}]\label{dff}
The stable double Grothendieck polynomial for $w$ is given by: 
\begin{eqnarray*}
 {\mathcal{G}}_{w}(\overleftarrow{x},\overleftarrow{y})=\mathfrak{G}_{\overrightarrow{w}}(x,y)|_{\overrightarrow{x}=0=\overrightarrow{y}}.
\end{eqnarray*}
\end{definition}
\begin{proposition}\label{squ}
\begin{eqnarray*}
 {\mathcal{G}}_{w}(\overleftarrow{x},\overleftarrow{y})=\sum_{\mathfrak{f}\in {\mathcal{F}}^{\square}_{w}} (\overleftarrow{x},\overleftarrow{y})^{wt(\mathfrak{f})}
\end{eqnarray*}
\end{proposition}
\begin{proof}
We have that
\begin{eqnarray*}
 \mathfrak{G}_{\overrightarrow{w}}(x,y)= \mathfrak{G}_{\overrightarrow{w}}^{\square}(x,y)=\sum_{\mathfrak{f}\in {\mathfrak{F}}^{\square}_{\overrightarrow{w}}} (x,y)^{wt(\mathfrak{f})}.
\end{eqnarray*}
Therefore we have that 
\begin{eqnarray*}
 \mathfrak{G}_{\overrightarrow{w}}(x,y)|_{\overrightarrow{x}=0=\overrightarrow{y}}=\sum_{\mathfrak{f}\in {\mathfrak{F}}^{\square}_{\overrightarrow{w}}(2m+2)} (\overleftarrow{x},\overleftarrow{y})^{wt(\mathfrak{f})}
\end{eqnarray*}
where ${\mathfrak{F}}^{\square}_{\overrightarrow{w}}(2m+2)$ is the subset of ${\mathfrak{F}}^{\square}_{\overrightarrow{w}}$ where all but the middle $2m+2$ factors are empty.  But every Hecke word for $\overrightarrow{w}$ only contains elements from the set of $\{(m+1),\ldots,(m+k)\}$ and the boundedness condition on the central $2m+2$ factors of a factorization of ${\mathfrak{F}}^{\square}_{\overrightarrow{w}}(2m+2)$ only requires elements to be greater $\geq i$ for some $i \leq m+1$. Therefore no factorization of ${\mathcal{F}}^{\square}_{w}$   fails to lie inside of  ${\mathfrak{F}}^{\square}_{\overrightarrow{w}}(2m+2)$  (after changing each $s_i$ to $s_{i+m}$), that is, ${\mathcal{F}}^{\square}_{w}={\mathfrak{F}}^{\square}_{\overrightarrow{w}}(2m+2)$   (after changing each $s_i$ to $s_{i+m}$).
\end{proof}

\subsection{Balanced Double Grothendieck Polynomials}
\begin{definition}\label{PSVT}
Consider the orderdered alphabet $\{1<2<\cdots<m+1<1'<2'<\cdots<(m+1)'\}$.  A
 \emph{primed set valued tableau}  of shape $\lambda$,  or an element of $PSVT(\lambda)$,
 is a filling of a Young diagram of shape $\lambda$ such that each box is nonempty and contains a set from this alphabet such that 
\begin{itemize}
\item All of the entries in a box are less than or equal to  all of the entries in the box to its right.
\item All of the entries in a box are less than or equal to  all of the entries in the box below it.
\item $i$ appears in at most one box in each column.
\item $i'$ appears in at most one box in each row.
\end{itemize} 
\end{definition}
The $x$-weight of such a tableau is the vector whose $i^{th}$ coordinate records the number of times $i'$ appears in the tableau. The $y$-weight is the vector whose $i^{th}$ coordinate records the number of times $i$ appears in the tableau.  

\begin{example} The following is a $PSVT$ with  $x$-weight $(3,3,2,2)$   and $y$-weight $(1,2,2,0)$
\begin{eqnarray*}
\Yboxdim{20pt}\young({{12}}{{23}}{{1'2'}},{{31'}}{{2'3'}}{{4'}},{{1'2'}}{{3'4'}}) 
\end{eqnarray*}
\end{example}
\begin{definition}
 We define the balanced double Grothendieck polynomial\footnote{Although, as we will see the stable double Grothendiecks will expand in terms of these polynomials, it should be noted that this does not imply that they are stable limits themselves (in general, they are not instantiations of stable double Grothendieck polynomials indexed by Grassmannian permutations).}:
\begin{eqnarray*}
G_{\lambda}(\overleftarrow{x},\overleftarrow{y})=\sum_{T \in PSVT(\lambda)} (\overleftarrow{x},\overleftarrow{y})^{wt(T)}
\end{eqnarray*}
\end{definition}

\subsection{Relationship between Stable Double Grothendieck  Polynomials and Balanced Double Grothendieck  Polynomials}
 We are interested now in the relationship between $\mathcal{G}_{w}(\overleftarrow{x},\overleftarrow{y})$ and $G_{\lambda}(\overleftarrow{x},\overleftarrow{y})$.
\begin{proposition}\label{flipn}
 There is an $x$-weight and $y$-weight preserving bijection from $\mathcal{F}^{\square}_{w}$ to pairs $(P,Q)$ where $P \in HT_{w}$ and $Q \in PSVT$ have the same shape.
\end{proposition}
\begin{proof}

Let $\mathfrak{f} \in {\mathcal{F}}_{w}^{\square}$  and let $\mathfrak{f}_{\ell}$ represent the leftmost $m+1$ factors of $\mathfrak{f}$ and $\mathfrak{f}_r$ represent the rightmost $m+1$ factors of $\mathfrak{f}$.  Suppose that $\mathfrak{f}_{\ell}$ represents the permutation $\mu$ and denote by $\overleftrightarrow{\mathfrak{f}_{\ell}}$ the factorization given by reversing the order of the factors of $\mathfrak{f}_{\ell}$ and reversing the order of the letters within each factor. Note that if $\mathfrak{f}_{\ell}$ is  a Hecke factorization for $\mu$ then   $\overleftrightarrow{\mathfrak{f}_{\ell}}$  is  a  Hecke factorization of  $\mu^{-1}$.  Apply semistandard Hecke insertion to $\overleftrightarrow{\mathfrak{f}_{\ell}}$ to obtain a pair $(P_{\ell},Q_{\ell})$ where $P_{\ell} \in HT_{\mu^{-1}}(\lambda_{\ell})$ and $Q_{\ell} \in CSVT(\lambda_{\ell})$ for some $\lambda_{\ell}$.  Now transpose both tableaux to get a pair $(P_{\ell}^t,Q_{\ell}^{t})$ of shape $\lambda_{\ell}^t$.  Now, proceed with semistandard Hecke insertion as if the current insertion tableau were $P_{\ell}^t$ and the current recording tableau were $Q_{\ell}^{t}$ and exactly the factors of $\mathfrak{f}_r$ remained to be inserted.  The only ambiguity to starting in the middle of Hecke insertion like this is not knowing what entry to add to the recording tableau during insertion of the $i^{th}$ factor of $\mathfrak{f}_r$: Use the entry $i'$.  Denote the final insertion tableau and recording  tableau as $P$ and $Q$ respectively.  We can now define the bijection: $\Phi(\mathfrak{f})=(P,Q)$.  
\begin{example}
Let $\mathfrak{f}=(124)(13)|(432)(3) \in {\mathcal{F}}_{w}^{\square}$.  we have $\mathfrak{f}_{\ell}=(124)(13)$ and $\mathfrak{f}_r=(432)(3)$. First apply semistandard Hecke insertion to $\overleftrightarrow{\mathfrak{f}_{\ell}}=(31)(421)$ to find that 
\begin{eqnarray*}
P_{\ell}=\young(12,24,3), \,\,\,\,\,\, Q_{\ell}=\young(12,12,2), \,\,\,\,\,\, P_{\ell}^t=\young(123,24), \,\,\,\,\,\, Q_{\ell}^{t}=\young({{1}}{{1}}{{2}},{{2}}{{2}})
\end{eqnarray*}
Now apply semistandard Hecke insertion of $\mathfrak{f}_r=(432)(3)$ to the starting pair $(P_{\ell}^t, Q_{\ell}^{t'})$
\begin{eqnarray*}
P=P_{\ell}^t \leftarrow (432)(3)=\young(123,24) \leftarrow (432)(3)= \young(1234,234,4)\\
Q=\young({{1}}{{1}}{{2}}{{1'}},{{2}}{{21'}}{{2'}},{{1'}})
\end{eqnarray*}
\end{example}
There is much to prove:
\begin{itemize}
\item $P$ is a Hecke tableau and it represents the permutation $w$:  Suppose $\mathfrak{f}_{\ell}$ is a Hecke word for some permutation   $\mu$. Now, $P_{\ell}$ was formed by applying Hecke insertion to $\overleftrightarrow{\mathfrak{f}_{\ell}}$   and so is a Hecke tableau whose rows read left to right, from bottom row to top row form a Hecke word for $\mu^{-1}$.  Since the only requirement for being a Hecke tableau is that the rows and columns are strictly increasing, (which is clearly preserved under transposition) it is also true that $P_{\ell}^t$ is a Hecke tableau.  Next, the columns of $P_{\ell}^t$ read from top to bottom from rightmost column to leftmost column give a Hecke word for $\mu^{-1}$.  Therefore the columns of $P_{\ell}^t$ read from bottom to top from leftmost column to rightmost column give a Hecke word for $\mu$.  However:
\begin{claim}
The column reading word and row reading word of a Hecke tableau, $H$, represent the same permutation.
\end{claim}
\begin{proof}
Let $w_k(H)$ be the permutation represented by reading the leftmost $k$ columns of $H$ bottom to top, leftmost column  to rightmost column and then, ignoring the first $k$ columns of $H$, reading rows left to right, bottom row to top row.  It suffices to show that $w_k(H)=w_{k+1}(H)$. Without loss of generality we may assume $k=0$.  Now let $w^j(H)$ be the permutation represented by reading the lowest $j$ entries of the leftmost column of $H$ from bottom to top and then reading the remaining entries of $H$ by rows, left to right, bottom to top.  To show that  $w_k(H)=w_{k+1}(H)$ for $k=0$ it suffices to show that $w^j(H)=w^{j+1}(H)$. If $a$ is the entry in the leftmost column of $H$ in the ${j+1}^{st}$ row from the bottom and $b$ is any entry in $H$ in the ${j}^{th}$ row from the bottom or lower not in the first column of $H$ then $a<b-1$.  Therefore $a$ commutes with all such $b$ which shows that $w^j(H)=w^{j+1}(H)$.
\end{proof}
Therefore reading the rows of  $P_{\ell}^t$ left to right, bottom to top also gives a Hecke word for $\mu$.  Since $\mathfrak{f}_r$ gives a Hecke word for some permutation $\nu$ such that $\widetilde{\mu}\widetilde{\nu} \sim \overrightarrow{w}$  the properties of Hecke insertion imply that the Hecke word formed by reading the rows of $P$ from left to right, bottom to top also represents $w$.  All this shows that $P$ is a Hecke tableau and it represents the permutation $w$.

\item $Q \in PSVT(\lambda)$ where $\lambda$ is the shape of $P$:  First, $Q_{\ell} \in CSVT(\lambda_{\ell})$ by Lemma \ref{SSbij} so it follows that $Q_{\ell}^{t} \in PSVT(\lambda_{\ell}^t)$ (and has no primed entries).  On the other hand it also follows from Lemma \ref{SSbij} that the primed entries from $Q$ will give (ignoring their primes) an element of $CSVT(\lambda/\rho)$ for some $\rho \subseteq \lambda_{\ell}^t$ such that $\lambda_{\ell}^t \setminus \rho$ contains no more than one box in any row or column.  The fact that the primed and unprimed entries give such tableaux along with the fact that $i < j'$ for any $i$ and $j$ imply that $Q \in PSVT(\lambda)$.
\item $\Phi$ is injective.  Let $\mathfrak{f},\mathfrak{f}^{\times} \in \mathcal{F}^{\square}_{w}$. Suppose that $\Phi(\mathfrak{f})=\Phi(\mathfrak{f}^{\times})$ with $\mathfrak{f} \neq \mathfrak{f}^{\times}$.  We use the same notation as in the construction of $\Phi(\mathfrak{f})$ and also set $\mathfrak{p}_{\ell}$ equal to the Hecke factorization given by reading the columns of $P_{\ell}^t$ bottom to top from left column to right column.    Use the same notation for corresponding objects associated to $\mathfrak{f}^{\times}$ but with a ${\times}$.  

If $\mathfrak{f}_{\ell} \neq \mathfrak{f}_{\ell}^{\times}$ then by Lemma \ref {SSbij} $(P_{\ell},Q_{\ell}) \neq (P_{\ell}^{\times},Q_{\ell}^{\times})$.  But $Q_{\ell} \neq Q_{\ell}^{\times}$ would force $Q \neq Q^{\times}$ so we must have $P_{\ell}\neq P_{\ell}^{\times}$ and so $P_{\ell}^t \neq (P_{\ell}^{\times})^t$.  Thus either $\mathfrak{f}_{\ell} \neq \mathfrak{f}_{\ell}^{\times}$  in which case  $\mathfrak{p}_{\ell} \neq \mathfrak{p}_{\ell}^{\times}$  or else $\mathfrak{f}_{r} \neq \mathfrak{f}_{r}^{\times}$. Either way,  $\mathfrak{p}_{\ell} \mathfrak{f}_r \neq \mathfrak{p}_{\ell}^{\times} \mathfrak{f}_r^{\times}$.  But it is easy to see that the insertion tableau of $\mathfrak{p}_{\ell}$ is just $P_{\ell}^t$ and the insertion tableau of   $\mathfrak{p}_{\ell}^{\times}$ is just ${P_{\ell}^t}^{\times}$. Meanwhile the recording tableaux of $\mathfrak{p}_{\ell}$ and $\mathfrak{p}_{\ell}^{\times}$ are the same.  Thus the Hecke factorizations $\mathfrak{p}_{\ell} \mathfrak{f}_r$ and  $\mathfrak{p}_{\ell}^{\times} \mathfrak{f}_r^{\times}$ would be two distinct elements mapping to the same insertion and recording tableaux under the bijection of Lemma \ref{SSbij} which is a contradiction.
\item $\Phi$ is surjective. Suppose we are given $(P,Q)$ of the same  shape $\lambda$ where $P \in HT_{w}$ and $Q \in PSVT$.  Let $Q_{out}$ denote the skew tableau formed by only taking the primed entries  of $Q$ and unpriming them. Let $Q_{in}$ denote the tableau formed by taking only the unprimed entries of $Q$.  Take $j$ sufficiently large, (for example more than the number of unprimed entries in $Q$) and  let $Q_{can}$ be any $CSVT$ such that erasing all integers less than or equal to $j$ and subtracting $j$ from the rest gives $Q_{out}$ and such that removing all entries greater than $j$ gives a tableau of the same shape as $Q_{in}$.

 Now use Lemma \ref{SSbij} to find a Hecke factorization $\widetilde{\mathfrak{f}}$ mapping to $(P,Q_{can})$.  Write  $\widetilde{\mathfrak{f}}=\widetilde{\mathfrak{f}_{\ell}}\mathfrak{f}_r$ where $\widetilde{\mathfrak{f}_{\ell}}$ represents the first $j$ factors of $\widetilde{\mathfrak{f}}$.  Suppose the insertion tableau of  $\widetilde{\mathfrak{f}_{\ell}}$ is $T$.  Use Lemma \ref{SSbij} to find a Hecke factorization $ \mathfrak{f}_{\ell}$ mapping to $(T^t,Q_{in}^t)$.  Let $\overleftrightarrow{\mathfrak{f}_{\ell}}$ represent the result of reversing the order of the factors of $\mathfrak{f}_{\ell}$ and reversing the order of the entries within each factor.  Then we have that $\Phi(\overleftrightarrow{\mathfrak{f}_{\ell}}\mathfrak{f}_{r})=(P,Q)$.  Now $\overleftrightarrow{\mathfrak{f}_{\ell}}\mathfrak{f}_{r} \in {\mathcal{F}}_{w'}^{\square}$ for some $w'$ just by construction.  But by the first bullet point we have $w'=w$.
\item $\Phi$ preserves the $x$-weight and the $y$-weight:  Suppose $\Phi(\mathfrak{f})=(P,Q)$ where $\mathfrak{f}=\mathfrak{f}_{\ell}\mathfrak{f}_r$.  The $y$-weight of $\mathfrak{f}$ is the vector whose $i^{th}$ coordinate records the number of entries in the $i^{th}$ factor of $\overleftrightarrow{\mathfrak{f}_{\ell}}$ which is the number of times $i$ appears in $Q_{\ell}$ or equivalently  in $Q$.  This is the definition of the $y$-weight of $Q$. The $x$-weight of $\mathfrak{f}$ is the vector whose $i^{th}$ coordinate records the number of entries in the $i^{th}$ factor of $\mathfrak{f}_r$ which is the number of times $i'$ appears in $Q$. This is the definition of the $x$-weight of $Q$. 
\end{itemize}
\end{proof}
\begin{remark}
If it were not for the unfortunate fact mentioned in Remark \ref{pita} the whole process of reversing the left side of the factorization and then inserting and then transposing would not be necessary and the proposition could be proved through just inserting the factors directly.  We leave it as an open problem to find a way of altering Hecke insertion so it has the additional properties needed for this simpler proof.
\end{remark}

\begin{corollary}\label{tab}
Given a partition $\mu$ we say that $\rho \sdot \mu$ if $\rho \subseteq \mu$ and $\mu/\rho$ contains no two boxes in the same row and no two boxes in the same column.  For a tableau, $T$, let $T_s$ denote the shape of $T$.  We have:
\begin{eqnarray} \label{tabf}
 {\mathcal{G}}_{w}(\overleftarrow{x},\overleftarrow{y})=
 \sum_{T \in HT_{w}} \sum_{\rho \sdot \mu \subseteq T_s} G_{T_s/\rho}(\overleftarrow{x})G_{\mu'}(\overleftarrow{y}).
\end{eqnarray}
\end{corollary}
\begin{proof}
It follows from proposition \ref{flipn} that we have:
\begin{eqnarray*}
 {\mathcal{G}}_{w}(\overleftarrow{x},\overleftarrow{y})=
\sum_{T \in HT_{w}} G_{T_s}(\overleftarrow{x},\overleftarrow{y}).
\end{eqnarray*}
Next, note that there is a canonical bijection from $PSVT(\lambda)$ to pairs of tableaux $(P,Q)$ where $P$ is a skew $CSVT$ and $Q^t$ is a straight shape $CSVT$ such that $P_s \cap Q_s$ contains no two boxes in the same row or column and where $P_s \cup Q_s =\lambda$.  Since the bijection sends $x$-weight to the weight of $P$ and $y$-weight to the weight of $Q$ the theorem follows from the formula above.  
\end{proof}

\subsection{Double Grothendieck functions}

We will now be interested in Grothendieck polynomials over infinite set(s) of variables.  To distinguish when we are talking about such polynomials will refer to them as functions from now on.  Let $\mathbf{x}=(x_1,x_2,\ldots)$ and  $\mathbf{y}=(y_1,y_2,\ldots)$ be  infinite lists of variables. Let ${\Omega}_{\mathbf{x}}$ be the $\mathbb{Z}[\mathbf{y}]$ linear involution on functions symmetric with respect to $\mathbf{x}$ in $\mathbb{Z}[\mathbf{x},\mathbf{y}]$ which sends $s_{\lambda}(\mathbf{x}) \rightarrow s_{\lambda'}(\mathbf{x})$.  Let ${\Omega}_{\mathbf{y}}$ be the $\mathbb{Z}[\mathbf{x}]$ linear involution on functions symmetric with respect to  $\mathbf{y}$ in $\mathbb{Z}[\mathbf{x},\mathbf{y}]$ that sends $s_{\lambda}(\mathbf{y}) \rightarrow s_{\lambda'}(\mathbf{y})$.  

In what follows  $\overline{CSVT(\lambda)}$, $\overline{\mathcal{F}_{w}}$,  $\overline{PSVT(\lambda)}$, and $ \overline{\mathcal{F}^{\square}_{w}}$ refer to the sets obtained by altering the definitions  of ${CSVT(\lambda)}$, ${\mathcal{F}_{w}}$,  ${PSVT(\lambda)}$, and ${\mathcal{F}^{\square}_{w}}$ respectively by replacing the requirement that entries come from the set $\{ 1,2,\ldots, m,m+1\}$ (and possibly their primed versions) with the requirement that the entries come from the set of positive integers (and possibly their primed versions).

\begin{definition}\cite{LS82}
Define the  conjugate symmetric Grothendieck function and the stable Grothendieck function, respectively, by:
\begin{eqnarray*}
 {{G}}_{\lambda}(\mathbf{x})=\sum_{T \in \overline{CSVT(\lambda)}} (\mathbf{x})^{wt(T)}\\
{\mathcal{G}}_{w}(\mathbf{x})=\sum_{\mathfrak{f}\in \overline{\mathcal{F}_{w}}} (\mathbf{x})^{wt(\mathfrak{f})}.\\
\end{eqnarray*}
\end{definition}

\begin{definition}\cite{LaPy07}
Define the  conjugate weak symmetric Grothendieck function and the weak stable Grothendieck function, respectively, by:
\begin{eqnarray*}
\prescript{*}{}{G_{\lambda}}(\mathbf{x})=\Omega_x ( G_{\lambda}(\mathbf{x}))\\
\prescript{*}{}{\mathcal{G}_{w}}(\mathbf{x})=\Omega_x(\mathcal{G}_{w}(\mathbf{x})).
\end{eqnarray*}
\end{definition}

\begin{definition}
Define the  balanced double Grothendieck function and the stable double Grothendieck function \cite{Buch02}, respectively, by:
\begin{eqnarray*}
G_{\lambda}(\mathbf{x},\mathbf{y})=\sum_{T \in \overline{PSVT(\lambda)}} (\mathbf{x},\mathbf{y})^{wt(T)}\\
 {\mathcal{G}}_{w}(\mathbf{x},\mathbf{y})=\sum_{\mathfrak{f}\in \overline{\mathcal{F}^{\square}_{w}}} (\mathbf{x},\mathbf{y})^{wt(\mathfrak{f})}.
\end{eqnarray*}
\end{definition}

\begin{definition}
Define the weak balanced double Grothendieck function and the weak stable double Grothendieck function, respectively, by:
\begin{eqnarray*}
\prescript{*}{}{G_{\lambda}}(\mathbf{x},\mathbf{y})=\Omega_x\Omega_y( G_{\lambda}(\mathbf{x},\mathbf{y}))\\
\prescript{*}{}{\mathcal{G}_{w}}(\mathbf{x},\mathbf{y})=\Omega_x\Omega_y(\mathcal{G}_{w}(\mathbf{x},\mathbf{y})).
\end{eqnarray*}
\end{definition}

This allows us to state the main theorem of this section:
\begin{theorem} \label{tabt}
Given a partition $\mu$ we say that $\rho \sdot \mu$ if $\rho \subseteq \mu$ and $\mu/\rho$ contains no two boxes in the same row and no two boxes in the same column.  For a tableau, $T$, let $T_s$ denote the shape of $T$.  We have:
\begin{eqnarray}
 {\mathcal{G}}_{w}(\mathbf{x},\mathbf{y})=
 \sum_{T \in HT_{w}} \sum_{\rho \sdot \mu \subseteq T_s} G_{T_s/\rho}(\mathbf{x})G_{\mu'}(\mathbf{y}) \label{first}\\
\prescript{*}{}{\mathcal{G}_{w}}(\mathbf{x},\mathbf{y})=
 \sum_{T \in HT_{w}} \sum_{\rho \sdot \mu \subseteq T_s} \prescript{*}{}{G_{T_s/\rho}}(\mathbf{x})\prescript{*}{}{G_{\mu'}}(\mathbf{y}) \label{second}.
\end{eqnarray}
\end{theorem}

\begin{proof}
Equation \ref{first} follows from  equation \ref{tabf}.  Equation \ref{second} follows by applying $\Omega_x \circ \Omega_y$ to equation \ref{first}.
\end{proof}

\begin{remark}\label{tabl}
Since $G_{\lambda}(\mathbf{x})$ can be defined combinatorially in terms of  set-valued tableaux of shape $\lambda'$ \cite{Buch02}, equation \ref{first} gives us a way to express  ${\mathcal{G}}_{w}(\mathbf{x},\mathbf{y})$ in terms of triples of tableaux, $(T,P,Q)$ where $T$ is a Hecke tableau and $P$ and $Q$ are (possibly skew) set-valued tableaux.    Similarly, since $\prescript{*}{}{G_{\lambda}}(\mathbf{x})$ can be defined combinatorially  in terms of  weak set-valued tableaux of shape $\lambda'$ \cite{LaPy07}, equation \ref{second} gives us a way to express  ${\mathcal{G}}_{w}(\mathbf{x},\mathbf{y})$ in terms of triples of tableaux, $(T,P,Q)$ where $T$ is a Hecke tableau and $P$ and $Q$ are (possibly skew) weak set-valued tableaux.  
\end{remark}

\section{Half weak double Grothendieck functions}\label{half}
\subsection{Motivation}  We begin by noting that there is a large degree of flexibility in the combinatorial models we have chosen in this paper.

First note that ${G}_{\lambda}(\mathbf{x},\mathbf{y})$ is expressed combinatorially in terms of primed set valued tableaux with entries from the infinite alphabet $\{1',2',\ldots,1,2,\ldots\}$.  On the other hand, it may be shown that $\prescript{*}{}{{G}_{\lambda}}(\mathbf{x},\mathbf{y})$ may be expressed in terms of ``primed \emph{multiset} valued tableaux" where boxes can now be filled with multisets from $\{1',2',\ldots,1,2,\ldots\}$.  Further, in both cases, the ordering of the alphabet $\{1',2',\ldots,1,2,\ldots\}$ is irrelevant as long as it is fixed.  We can, for instance, assume the order $1'<1<2'<2<\cdots$ in either definition.  This can be proven using a generalization of the maps $\nearrow$ and $\searrow$ from Lemma 2.2 of  \cite{Hawkes.2022} to the set valued and multiset valued cases respectively.

As a side note,  this freedom of ordering implies that the functions ${{G}}_{\lambda}(\mathbf{x},\mathbf{x})$ are actually certain skew $GQ$ functions as defined in \cite{IN13}: Precisely, suppose that the partition $\lambda$ has $k$ parts and let $\delta=(k-1,k-2,\ldots,1)$ be the staircase partition.    Then ${{G}}_{\lambda}(\mathbf{x},\mathbf{x})$ is the  K-theoretic $GQ$ function of \cite{IN13} indexed by the skew shifted   shape $(\lambda+\delta)/\delta$\footnote{Here  $\lambda+\delta$ means the strict partition obtained by adding $\delta$ to $\lambda$ coordinate-wise.  Since $\delta$ itself is a strict partition we can consider both $\delta$ and  $\lambda+\delta$ as shifted shapes and   $(\lambda+\delta)/\delta$ as a skew shifted shape.   This allows us to consider the straight-shape partition $\lambda$ as  a skew shifted shape.}. 
This can be seen by noting that primed set valued tableaux of shape $\lambda$ using the particular ordering of $1'<1<2'<2<\cdots$ are precisely the set-valued shifted tableaux of \cite{IN13} of shape $(\lambda+\delta)/\delta$.  Interestingly, it appears that ${{G}}_{\lambda}(\mathbf{x},\mathbf{x})$ expands in terms of non-skew $GQ$ functions and that this fact is a result of a more general phenomenon.   In Conjecture \ref{conj} of Section \ref{open} we give a precise statement  of what we mean by this (see also  Conjecture 5.14 of \cite{LM21}).


On the other hand, $\mathcal{G}_{w}(\mathbf{x},\mathbf{y})$ may be expressed combinatorially in terms of (unbounded) double Hecke factorizations where we allow an infinite number of factors to the left and to the right.  Similarly, it may be shown that $\prescript{*}{}{\mathcal{G}_{w}}(\mathbf{x},\mathbf{y})$  has such a combinatorial interpretation if we change the strictly increasing and strictly decreasing requirements on the left and right hand factors, respectively, to weakly decreasing and weakly increasing, respectively.  Again there is freedom in how we choose to define  these double Hecke factorizations.  Namely, the relative order in which the (weakly) increasing and (weakly) decreasing factors appear is not important as long as it is fixed.  For example, instead of requiring the increasing factors of an (unbounded) double Hecke factorization to all appear to the left of the decreasing factors, we may instead require they appear to the right of the decreasing factors, or, we could require that increasing and decreasing factors alternate, or, any other fixed arrangement of increasing and decreasing factors. The situation is the same for the weak case. These statements can be proven using the maps $\uparrow$ and $\downarrow$ of section 3 along with an adaptation of these maps to the weak case.

 If we were to choose to define (unbounded) double Hecke factorizations using the requirement that we must alternate between (strictly) decreasing and (strictly) increasing factors we would in essence end up with a definition  in terms of ``strict hook factors" where the decreasing part of each factor contributes to the $x$-weight and the increasing part to the $y$-weight. This is precisely the definition of type $B$ Stanley symmetric functions in \cite{Billey.Haiman.1995} except that those authors have only one weight per hook factor (which is the total length of that hook factor) and they only consider reduced words. Since $\mathcal{G}_{w}(\mathbf{x},\mathbf{x})$ also computes the weight this way it agrees with the type $B$ Stanley symmetric function on terms of lowest degree (since the terms of lowest degree in $\mathcal{G}_{w}(\mathbf{x},\mathbf{x})$ represent precisely the reduced words for $w$).  As an example if $w =(\mathbf{3}, \mathbf{2},\mathbf{1})$ then (we leave it to the reader to compute) that $\mathcal{G}_{w}(x_1,x_2,x_1x_2)=2x_1^3+8x_1^2x_2+8x_1x_2^2 +2x_2^2+8x_1^3x_2+16x_1^2x_2^2+8
x_1x_2^3+\cdots$ whereas the type $ B$ Stanley symmetric function is just $2x_1^3+8x_1^2x_2+8x_1x_2^2 +2x_2^2$.

 However, since  $\prescript{*}{}{\mathcal{G}_{w}}(\mathbf{x},\mathbf{y})$ can be interpreted combinatorially using ``weak" (unbounded) double Hecke factorizations with the requirement that we must alternate between (weakly) decreasing and (weakly) increasing factors we see that   $\prescript{*}{}{\mathcal{G}_{w}}(\mathbf{x},\mathbf{y})$ has a similar combinatorial interpretation in terms of ``weak hook factorizations."  Now, any weak hook factorization corresponding to a reduced word is  automatically also a strict hook factorization (and vice versa).  Thus $\prescript{*}{}{\mathcal{G}_{w}}(\mathbf{x},\mathbf{x})$ also agrees with the type $B$ Stanley symmetric function on terms of lowest degree.

Finally,  we could also  define a new function $\prescript{\times}{}{\mathcal{G}_{w}}(\mathbf{x},\mathbf{y})$  using ``half weak" hook factorizations that are composed of factors that are strictly decreasing and then weakly increasing.  By similar reasoning as above we see that $\prescript{\times}{}{\mathcal{G}_{w}}(\mathbf{x},\mathbf{x})$ would also generalize the type $B$ Stanley symmetric function in the sense that its lowest degree term would return the latter.    Moreover, there is a natural interpolation of ${G}_{w}(\mathbf{x},\mathbf{y})$ and $\prescript{*}{}{{G}_{w}}(\mathbf{x},\mathbf{y})$, which we denote $\prescript{\times}{}{{G}_{w}}(\mathbf{x},\mathbf{y})$, such that $\prescript{\times}{}{\mathcal{G}_{w}}(\mathbf{x},\mathbf{y})$ expands in terms of  $\prescript{\times}{}{{G}_{w}}(\mathbf{x},\mathbf{y})$ with non-negative integer coefficients.  Further, the expansion coefficients are the same as those in the expansion of    $\prescript{}{}{\mathcal{G}_{w}}(\mathbf{x},\mathbf{y})$  in terms of   $\prescript{}{}{{G}_{w}}(\mathbf{x},\mathbf{y})$ (and so also as in the expansion of $\prescript{*}{}{\mathcal{G}_{w}}(\mathbf{x},\mathbf{y})$  in terms of   $\prescript{*}{}{{G}_{w}}(\mathbf{x},\mathbf{y})$).  In fact the relation  between  $\prescript{\times}{}{\mathcal{G}_{w}}(\mathbf{x},\mathbf{y})$ and  $\prescript{\times}{}{{G}_{w}}(\mathbf{x},\mathbf{y})$ appears more naturally in the sense that it can be proven directly using Hecke insertion (remark \ref{pita} does not cause issues in this case).  However, the most compelling evidence that $\prescript{\times}{}{{\mathcal{G}}_{w}}(\mathbf{x},\mathbf{x})$ is a more natural generalization of the type $B$ Stanley symmetric function than ${\mathcal{G}}_{w}(\mathbf{x},\mathbf{x})$ or $\prescript{*}{}{\mathcal{G}_{w}}(\mathbf{x},\mathbf{x})$  is that, like the type $B$ Stanley symmetric function,  $\prescript{\times}{}{{\mathcal{G}}_{w}}(\mathbf{x},\mathbf{x})$ is $Q$-Schur positive whereas neither ${\mathcal{G}}_{w}(\mathbf{x},\mathbf{x})$ nor $\prescript{*}{}{\mathcal{G}_{w}}(\mathbf{x},\mathbf{x})$  is.

For these reasons it seems like $\prescript{\times}{}{\mathcal{G}_{w}}(\mathbf{x},\mathbf{x})$ is a suitable candidate for  a type $B$ stable Grothendieck function.  Of course, our definition is incomplete in the sense that it is only defined for (unsigned) permutations whereas we would ideally like something defined  for all signed permutations. This is equivalent to adding a rule as to how the special generator, $s_0$,  of the type $B$ Weyl group should be incorporated into the definition of hook Hecke factorization found in the next subsection.  We leave this as an open problem.

\subsection{Results}
We need to introduce a number of definitions:

\begin{definition}
A \emph{hook Hecke factorization} of $w$ is a factorization into hook factors.  Each hook factor contains a subset of $\{\circled{1},\circled{2},\ldots\}$ and a multiset from $\{1,2,\ldots\}$ arranged so that all circled factors lie to the left of all uncircled factors and such that the circled elements are strictly decreasing left to right and the uncircled elements are weakly increasing left to right. Moreover, erasing the circles and parentheses should give a Hecke word for $w$. For instance, $\big(\circled{3}\circled{2}{2}{3}{3}\big)\big(\circled{1}{2}{2}\big)\big(\circled{3}\circled{2}1133\big)$ is a hook Hecke factorization for the permutation $(4,3,2,1) \in S_4$. The \emph{$x$-weight}   of  a  hook Hecke factorization is the vector whose $i^{th}$ entry records the number of uncircled elements in its $i^{th}$ factor. The $x$-weight of the example above is $(3,2,4)$.   The \emph{$y$-weight} of such a factorization is the vector whose $i^{th}$ entry records the number of circled entries in the $i^{th}$ factor.  The $y$-weight of the example above is $(2,1,2)$. 
\end{definition}

\begin{remark}\label{symmetry}
There is a symmetry that must be broken as to the assignment of the $x$-weight and $y$-weight.  While the $x$-weight of $\mathcal{G}_{w}(\mathbf{x},\mathbf{y})$ corresponds to strictly decreasing factors, the $x$-weight of  $\prescript{*}{}{\mathcal{G}_{w}}(\mathbf{x},\mathbf{x})$ corresponds to weakly increasing factors.  Thus either half of the  hooks could naturally be chosen as the $x$-weight of a hook Hecke factorization.  We break the symmetry in such a way that agrees with \cite{Hawkes.2022}.
\end{remark}

\begin{definition} Denote the set of all hook Hecke factorizations of $w$  by $\prescript{\times}{}{\mathcal{F}}_{w}$.  Define the half weak stable double Grothendieck function by
\begin{eqnarray*}
\prescript{\times}{}{\mathcal{G}}_{w}(\mathbf{x},\mathbf{y})=\sum_{\mathfrak{f}\in \prescript{\times}{}{\mathcal{F}}_{w}} (\mathbf{x},\mathbf{y})^{wt(\mathfrak{f})}.
\end{eqnarray*}
\end{definition}

\begin{definition}  Consider the ordered alphabet $1'<1<2'<2<\cdots$. A
 \emph{primed special multiset tableau}  of shape $\lambda$,  or an element of $PSMT(\lambda)$,
 is a filling of a Young diagram of shape $\lambda$ such that each box is nonempty and contains a multiset from this alphabet such that 
\begin{itemize}
\item All of the entries in a box are less than or equal to  all of the entries in the box to its right.
\item All of the entries in a box are less than or equal to  all of the entries in the box below it.
\item $i$ appears in at most one box in each column.
\item $i'$ appears in at most one box in each row.
\item Each box contains at most one $i'$.
\end{itemize} 
\end{definition}

\begin{definition} 
If $T \in PSMT(\lambda)$ then the $x$-weight\footnote{The assignment is the reverse of that which appears in definition \ref{PSVT} as a result of the way we have broken the symmetry referred to in remark \ref{symmetry}.} of $T$  is the vector whose $i^{th}$ entry records the number of instances of $i$ in $T$ and the $y$-weight of $T$  is the vector whose $i^{th}$ entry records the number of instances of $i'$ in $T$.   Define the half weak balanced double Grothendieck function by
\begin{eqnarray*}
\prescript{\times}{}{{G}}_{w}(\mathbf{x},\mathbf{y})=\sum_{T \in PSMT(\lambda)} (\mathbf{x},\mathbf{y})^{wt(T)}.
\end{eqnarray*}
\end{definition}

\begin{example} A primed special multiset tableau,  $Q \in {PSMT}(3,3,2)$, with $x$-weight of $(3,2,3,0,0,\ldots)$ and $y$-weight of $(1,3,3,0,0,\ldots)$ is shown below.

\begin{eqnarray*}
Q=\Yboxdim{24pt}
\young({{1'11}}{{12'}}{{23'}},{{2'}}{{2}}{{3'33}},{{2'3'}}{{3}})
\end{eqnarray*}
\end{example}

The following lemma and its corollary relates these definitions.

\begin{lemma}\label{SMbij}
 There is an $x$-weight and $y$-weight preserving bijection from $\prescript{\times}{}{\mathcal{F}}_{w}$ to pairs $(P,Q)$ where $P \in HT_{w}$ and $Q \in PSMT$ have the same shape.    Here the $x$-weight and $y$-weight  of $(P,Q)$ are defined as the $x$-weight and $y$-weight of $Q$.  
\end{lemma}
\begin{proof}
A hook Hecke factorization is just a Hecke word, $\alpha$, along with an ordered set partition of $\{1,2,\ldots,|\alpha|\}$ into parts such that
\begin{itemize}
\item Each part contains consecutive numbers and parts with smaller numbers precede parts with larger numbers.
\item If $a$ and $a+1$ occur in the same one of one of parts number $1,3,5,\ldots$ then, $\alpha_a > \alpha_{a+1}$.
\item If $a$ and $a+1$ occur in the same one of one of  parts number $2,4,6,\ldots$ then, $\alpha_a \leq \alpha_{a+1}$.
\end{itemize}
On the other hand using a standardization argument we see that a $PSMT$ is just a  standard set valued tableau, $T$, along with a set partition of $\{1,2,\ldots, max(T)\}$ into parts such that
\begin{itemize}
\item Each part contains consecutive numbers and parts with smaller numbers precede parts with larger numbers.
\item If $a$ and $a+1$ occur in the same one of one of  parts number $1,3,5,\ldots$ then $a+1$ lies strictly below $a$ in $T$.
\item If $a$ and $a+1$ occur in the same one of one of  parts number $2,4,6,\ldots$ then $a+1$ lies either in the same box as $a$ or strictly right of $a$ in $T$.
\end{itemize}
Lemma \ref{combined} implies that if $\alpha \rightarrow (P,Q)$ under the bijection of Proposition \ref{Hebij} then a given set partition of $\{1,2,\ldots, |\alpha|\}$  turns $\alpha$ into a hook Hecke factorization if and only if the same set partition turns $Q$ into a $PSMT$.  Thus combining the bijection of Proposition \ref{Hebij} with the identity on set partitions induces the weight preserving bijection of the lemma. 
\end{proof}

\begin{corollary}\label{ex1}
Letting $H_{w}^{\rho}$ denote the number of Hecke tableaux for $w$ with shape $\rho$ we have
\begin{eqnarray*}
 \prescript{\times}{}{\mathcal{G}}_{w}(\mathbf{x},\mathbf{y})=
 \sum_{\rho}( H_{w}^{\rho}) \prescript{\times}{}G_{\rho}(\mathbf{x},\mathbf{y}).
\end{eqnarray*}
\end{corollary}

\begin{proof}
This follows from the lemma above.  
\end{proof}

Next, we need to define another couple of types of tableaux:

\begin{definition}
Let $\mu \subseteq \lambda$ be partitions with an equal number of rows. A \emph{flagged reverse semistandard tableau} of shape $\lambda/\mu$, or an element of ${FRSST}(\lambda/\mu)$, is a filling of a Young diagram of shape $\lambda/\mu$ using the alphabet $1<2<\cdots$ such that:
\begin{itemize}
\item Each box in row $i$ of $\lambda/\mu$ contains one element from $\{1,2,\ldots,\mu_i\}$.
\item The rows are weakly decreasing from left to right.
\item The columns are strictly decreasing from top to bottom.  
\end{itemize}
On the other hand, if $\lambda$ contains more rows than $\mu$, we define ${FRSST}(\lambda/\mu)=\emptyset$.  
\end{definition}

\begin{example} An $FRSST$  of shape  $(6,6,5,4)/(4,3,2,1))$ with the inner shape shown filled with grey boxes is shown below.
\begin{center}
\ytableausetup{centertableaux}
\ytableaushort
{\none \none \none \none 42, \none \none \none 321, \none \none 221, \none 111}
* {6,6,5,4}
* [*(lightgray)]{4,3,2,1}
\end{center}

Note that the definition requires that the maximum number in row $i$ is no greater than the number of shaded boxes in that row (which is $\mu_i$).

\end{example}

\begin{definition}
A primed tableau of shape $\lambda$, or element of $PT(\lambda)$ is an element of $PSMT(\lambda)$ with exactly one entry in each box.  
\end{definition}

\begin{definition}\label{}
The double $Q$-Schur function is defined as:
\begin{eqnarray*}
R_{\lambda}(\mathbf{x},\mathbf{y})=\sum_{T \in PT(\lambda)} (\mathbf{x},\mathbf{y})^{wt(T)}.
\end{eqnarray*}
\end{definition}
The relationship between $PSMT$, $PT$, and $FRSST$ (called $OFT$ therein) is given by part (1) of Lemma 1.9 of \cite{Hawkes.2022}.
\begin{lemma}[\cite{Hawkes.2022}]
 There is an $x$-weight and $y$-weight preserving bijection from $PSMT(\mu)$ to pairs of tableaux $(P,Q)$ where $P \in PT(\lambda)$ and $Q \in FRSST(\lambda/\mu)$ for some $\lambda \supseteq \mu$.  Here the $x$-weight and $y$-weight  of $(P,Q)$ are defined as the $x$-weight and $y$-weight of $P$.  
\end{lemma}
A corollary of this lemma relates  $\prescript{\times}{}{\mathcal{G}}_{w}(\mathbf{x},\mathbf{y})$ to
$R_{\mu}(\mathbf{x},\mathbf{y}).$
\begin{corollary}\label{cor}
Letting $K_{\rho}^{\mu}$ be the number of $FRSST$ of shape $\mu/\rho$ we have
\begin{eqnarray*}
 \prescript{\times}{}{\mathcal{G}}_{w}(\mathbf{x},\mathbf{y})=
 \sum_{\rho \subseteq \mu}  H_{w}^{\rho} K_{\rho}^{\mu}  R_{\mu}(\mathbf{x},\mathbf{y}).
\end{eqnarray*}
\end{corollary}
\begin{proof}
This follows from the lemma above along with corollary \ref{ex1}.  
\end{proof}

Finally, we need to identify certain special elements of $PT(\lambda)$. (The following two definitions are more general than they need to be for this purpose alone. They are also used in Conjecture \ref{conj}.)

\begin{definition}\label{starting}
Let $T$ be any tableau whose entries come from the alphabet $\{1',2',\ldots,1,2,\ldots\}$.  If $T$ has  boxes with multiple entries, order the entries horizontally within those boxes as follows: Starting from the left and moving right,  list the unprimed entries in increasing order followed by the primed entries in decreasing order.  

Start with your  finger in the leftmost box of the lowest row of $T$.  Now move your finger left to right across rows moving from bottom row to top row stopping once your finger  lies over an instance of $i$ or $i'$ for the first time. If your finger lies over an $i'$ we say that $T$ has the \emph{$i$-primed property}.  If $T$ has no $i$ or $i'$ we also say that $T$ (trivially) has the \emph{$i$-primed property}.
\end{definition}

\begin{definition}\label{lattice}
Let $T$ be any tableau whose entries come from the alphabet $\{1',2',\ldots,1,2,\ldots\}$. If $T$ has  boxes with multiple entries, order the entries horizontally within those boxes as follows: Starting from the left and moving right,  list the unprimed entries in increasing order followed by the primed entries in decreasing order.  Let $\mathtt{prop}$ be a Boolean initialized to $\mathtt{true}$.
\begin{itemize}
\item Place your  finger in the rightmost box of the top row of $T$.  Drag your finger right to left moving down a row each time you get to the leftmost box of a row until you reach the bottom left box of $T$.  While you are scanning, any time your finger lies over an $i$ place a tally mark above $T$.  Any time your finger lies over an $i-1$ place a tally mark under $T$.  If there are ever more tally marks above $T$ than below $T$, set $\mathtt{prop}=\mathtt{false}$ and terminate the algorithm. If there are ever an equal number of tally marks above $T$ and below $T$ and your finger lies on an $i'$ also  set $\mathtt{prop}=\mathtt{false}$ and terminate the algorithm.
\item Next  (do NOT erase the tally marks from the last step), start with your finger in the leftmost box of the lowest row of $T$.  Now move your finger left to right across rows moving from bottom row to top row. While you are scanning, any time your finger lies over an $i'$ place a tally mark above $T$.  Any time your finger lies over an $(i-1)'$ place a tally mark under $T$. If there are ever more tally marks above $T$ than below $T$,  set $\mathtt{prop}=\mathtt{false}$ and terminate the algorithm.  If there are ever an equal number of tally marks above $T$ and below $T$ and your finger lies on an $i-1$ also  set $\mathtt{prop}=\mathtt{false}$ and terminate the algorithm.
\end{itemize}
We say that $T$ has the \emph{$i$-lattice property} if at the end of this process  $\mathtt{prop}=\mathtt{true}$.
\end{definition}

\begin{example} An element $P \in PT(6,4,4,4,3)$ is shown below.
\begin{eqnarray*}
\Yboxdim{14pt}
\young({{1'}}11111,1{{2'}}22,{{2'}}2{{3'}}3,2{{3'}}34,3{{4'}}4)
\end{eqnarray*}
Note that:
\begin{multicols}{2}
\begin{itemize}
\item $P$ lacks the $1$-primed property.
\item $P$ lacks the $2$-primed property.
\item $P$ lacks the $3$-primed property.
\item $P$  has the $4$-primed property.
\newline
\newline
\\
\item $P$ has the $1$-lattice property (trivially as there are no $0$s).
\item $P$ has the $2$-lattice property.
\item $P$ has the $3$-lattice property
\item $P$ does not have the $4$-lattice property: You set $\mathtt{prop}=\mathtt{false}$ when your finger first passes over the $4'$.
\end{itemize}
\end{multicols}
\end{example}

These definitions allow us to state a  particular case of Theorem 8.3 of \cite{Stembridge89}:

\begin{proposition}[\cite{Stembridge89}]\label{lat}
 We have:
\begin{eqnarray*}
R_{\mu}(\mathbf{x},\mathbf{x})=\sum_{\lambda} F_{\mu}^{\lambda} Q_{\lambda}(\mathbf{x})
\end{eqnarray*}
where $F_{\mu}^{\lambda}$ is the number of elements $T \in PT(\mu)$ that have $i$-lattice property and the $i$-primed property for all $i$ and such that the sum of the $x$-weight and $y$-weight of $T$ is equal to $\lambda$.
\end{proposition}

Combining Corollary \ref{cor} with Proposition \ref{lat}  will now give us the $Q$-Schur expansion of $\prescript{\times}{}{\mathcal{G}}_{w}(\mathbf{x},\mathbf{x})$:

\begin{theorem}  \label{QP}
The function $\prescript{\times}{}{\mathcal{G}}_{w}$ evaluated at $\mathbf{x}=\mathbf{y}$ is $Q$-Schur positive and:
\begin{eqnarray*}
 \prescript{\times}{}{\mathcal{G}}_{w}(\mathbf{x},\mathbf{x})=  \sum_{(\rho \subseteq \mu), \lambda} H_{w}^{\rho} K_{\rho}^{\mu}  F_{\mu}^{\lambda} Q_{\lambda}(\mathbf{x}).
\end{eqnarray*}
\end{theorem}

\begin{example}
Consider the permutation $w=(2,3,1,5,4) \in S_5$ and suppose that we are interested in computing the degree $4$ part of $\prescript{\times}{}{\mathcal{G}}_{(2,3,1,5,4)}(\mathbf{x},\mathbf{x})$.

To do this we first compute the three elements of $HT_{w}$ shown in the top of the diagram below.  For each such tableau, $H$, we compute all elements of $FRSST(\mu/ H_s)$ for all possible partitions $\mu \supseteq H_s$ that have $4$ boxes (since we are concerned with the degree $4$ part).  These are shown in the middle of the diagram below.  Finally, for each such tableau, $O$, we compute all elements of $PT(O_s)$ that have the $i$-primed property and the $i$-lattice property for all $i$.  The resulting tableaux are shown in the last two lines of the diagram.

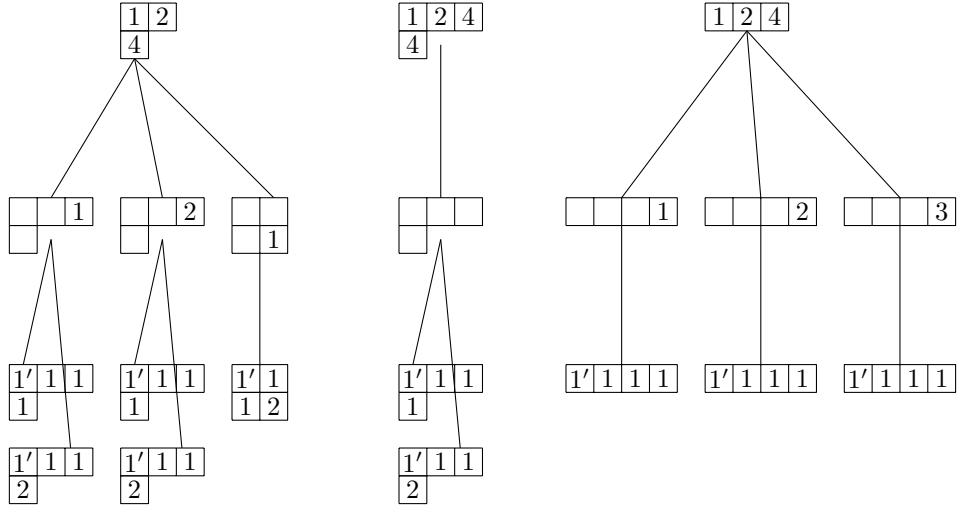
\begin{eqnarray*}
 \begin{tikzpicture}[scale=0.37]
 \draw[step=1cm] (4,16) grid (6,17); 
 \draw[step=1cm] (4,15) grid (5,16); 
\node at (4.5,15.5) {$4$};
\node at (4.5,16.5) {$1$};
\node at (5.5,16.5) {$2$};

 \draw[step=1cm] (14,15) grid (15,16); 
 \draw[step=1cm] (14,16) grid (17,17); 
\node at (14.5,15.5) {$4$};
\node at (14.5,16.5) {$1$};
\node at (15.5,16.5) {$2$};
\node at (16.5,16.5) {$4$};

 \draw[step=1cm] (25,16) grid (28,17); 
\node at (25.5,16.5) {$1$};
\node at (26.5,16.5) {$2$};
\node at (27.5,16.5) {$4$};

 \draw[step=1cm] (0,8) grid (1,9); 
 \draw[step=1cm] (0,9) grid (3,10); 
\node at (2.5,9.5) {$1$};

 \draw[step=1cm] (4,8) grid (5,9); 
 \draw[step=1cm] (4,9) grid (7,10); 
\node at (6.5,9.5) {$2$};

 \draw[step=1cm] (8,8) grid (10,9); 
 \draw[step=1cm] (8,9) grid (10,10); 
\node at (9.5,8.5) {$1$};

 \draw[step=1cm] (14,8) grid (15,9); 
 \draw[step=1cm] (14,9) grid (17,10); 

 \draw[step=1cm] (20,9) grid (24,10); 
\node at (23.5,9.5) {$1$};

 \draw[step=1cm] (25,9) grid (29,10); 
\node at (28.5,9.5) {$2$};

 \draw[step=1cm] (30,9) grid (34,10); 
\node at (33.5,9.5) {$3$};

 \draw[step=1cm] (0,2) grid (1,3); 
 \draw[step=1cm] (0,3) grid (3,4);
\node at (0.5,2.5) {$1'$};
\node at (0.5,3.5) {$1'$};
\node at (1.5,3.5) {$1$};
\node at (2.5,3.5) {$1$};

 \draw[step=1cm] (4,2) grid (5,3); 
 \draw[step=1cm] (4,3) grid (7,4); 
\node at (4.5,2.5) {$1'$};
\node at (4.5,3.5) {$1'$};
\node at (5.5,3.5) {$1$};
\node at (6.5,3.5) {$1$};

 \draw[step=1cm] (8,2) grid (10,3); 
 \draw[step=1cm] (8,3) grid (10,4); 
\node at (8.5,2.5) {$1'$};
\node at (8.5,3.5) {$1'$};
\node at (9.5,3.5) {$1$};
\node at (9.5,2.5) {$2'$};

 \draw[step=1cm] (14,2) grid (15,3); 
 \draw[step=1cm] (14,3) grid (17,4); 
\node at (14.5,2.5) {$1'$};
\node at (14.5,3.5) {$1'$};
\node at (15.5,3.5) {$1$};
\node at (16.5,3.5) {$1$};

 \draw[step=1cm] (20,3) grid (24,4); 
\node at (20.5,3.5) {$1'$};
\node at (21.5,3.5) {$1$};
\node at (22.5,3.5) {$1$};
\node at (23.5,3.5) {$1$};

 \draw[step=1cm] (25,3) grid (29,4); 
\node at (25.5,3.5) {$1'$};
\node at (26.5,3.5) {$1$};
\node at (27.5,3.5) {$1$};
\node at (28.5,3.5) {$1$};

 \draw[step=1cm] (30,3) grid (34,4); 
\node at (30.5,3.5) {$1'$};
\node at (31.5,3.5) {$1$};
\node at (32.5,3.5) {$1$};
\node at (33.5,3.5) {$1$};

 \draw[step=1cm] (0,-1) grid (1,0); 
 \draw[step=1cm] (0,0) grid (3,1);
\node at (0.5,-0.5) {$2'$};
\node at (0.5,0.5) {$1'$};
\node at (1.5,0.5) {$1$};
\node at (2.5,0.5) {$1$};

 \draw[step=1cm] (4,-1) grid (5,0); 
 \draw[step=1cm] (4,0) grid (7,1); 
\node at (4.5,-0.5) {$2'$};
\node at (4.5,0.5) {$1'$};
\node at (5.5,0.5) {$1$};
\node at (6.5,0.5) {$1$}; 

\draw[step=1cm] (14,-1) grid (15,0); 
 \draw[step=1cm] (14,0) grid (17,1); 
\node at (14.5,-0.5) {$2'$};
\node at (14.5,0.5) {$1'$};
\node at (15.5,0.5) {$1$};
\node at (16.5,0.5) {$1$};

\draw (4.5,15)--(1.5,10);
\draw (1.5,8.5)--(0.5,4);
\draw (1.5,8.5)--(2.2,1);

\draw (4.5,15)--(5.5,10);
\draw (5.5,8.5)--(4.5,4);
\draw (5.5,8.5)--(6.2,1);

\draw (4.5,15)--(9.5,10);
\draw (9,8)--(9,4);

\draw (15.5,15.5)--(15.5,10);

\draw (15.5,8.5)--(14.5,4);
\draw (15.5,8.5)--(16.2,1);

\draw (26.5,16)--(22,10);
\draw (26.5,16)--(27,10);
\draw (26.5,16)--(32,10);

\draw (32,9)--(32,4);
\draw (27,9)--(27,4);
\draw (22,9)--(22,4);

\end{tikzpicture}
\end{eqnarray*}

After counting the tableaux appearing in the last two lines of the digram above and computing their weights we see by Theorem \ref{QP} that the degree $4$ part of $\prescript{\times}{}{\mathcal{G}}_{(3,1,2,5,4)}(\mathbf{x},\mathbf{x})$ is equal to $6Q_{(4,0)}(\mathbf{x})+4Q_{(3,1)}(\mathbf{x})$.  In particular, the coefficient of $x_1^4x_2^0x_3^0\cdots$ in this expression is $12$ since this coefficient is $2$ in $Q_{(4,0)}(\mathbf{x})$ and $0$ in $Q_{(3,1)}(\mathbf{x})$.  This implies that there ought to be exactly 12 hook Hecke factorizations of $w$ which are composed of only one factor which has length $4$.  Indeed the following are all such factorizations:
\begin{eqnarray*}
\big(1124\big),\,\,\,\,\,
\big(1224\big),\,\,\,\,\,
\big(1244\big),\,\,\,\,\,
\big(\circled{4}112\big),\,\,\,\,\,
\big(\circled{4}122\big),\,\,\,\,\,
\big(\circled{4}124\big),\,\\
\big(\circled{1}124\big),\,\,
\big(\circled{1}224\big),\,\,
\big(\circled{1}244\big),\,\,
\big(\circled{4}\circled{1}12\big),\,\,
\big(\circled{4}\circled{1}22\big),\,\,
\big(\circled{4}\circled{1}24\big).
\end{eqnarray*}

\end{example}

\section{Conjectures and Open Problems}\label{open}

We end this paper with two open problems and a conjecture that have come up in our treatment of double Grothendieck polynomials.

\begin{open problem} The first open problem is to reformulate Hecke insertion so that it does not suffer from the drawback mentioned in remark \ref{pita} (while still maintaining the properties of Lemma \ref{combined}).  More explicitly, the problem is to define an insertion algorithm that gives a bijection from Hecke words for a permutation $w$ to pairs $(P,Q)$ of tableaux of the same shape where $P$ is a Hecke tableau for $w$ and $Q$ is a standard set-valued tableau such that the algorithm has the following property:  
 Suppose the word $\alpha_1\cdots \alpha_m$ maps to $(P,Q)$. Then 
\begin{itemize}
\item $\alpha_i>\alpha_{i+1}$ if and only if $i+1$ shows up in a row strictly below $i$ in $Q$.
\item $\alpha_i < \alpha_{i+1}$ if and only if $i+1$ shows up in a column to the right of  $i$ in $Q$.
\end{itemize}

\end{open problem}

\begin{open problem}
The second open problem is to generalize our definition of hook Hecke factorization to include signed permutations that have one or more instances of the generator $s_0$ in a reduced word for said signed permutation.  This should be done in such a way that the resulting generating function is symmetric, $Q$-Schur positive, and agrees with the type $B$ Stanley symmetric function on terms of lowest degree. The author has  tried a few of the more obvious ways of doing this, and so far none have been successful.  For instance, the most obvious approach--of simply extending the definition to include $s_0$ quickly fails:  For instance, the signed permutation $s_0s_1s_0$ has 8 hook Hecke factorizations of weight $(3,1)$:
\begin{eqnarray*}
\big(001\big)\big(0\big),
\big(001\big)\big(\circled{0}\big),
\big(\circled{0}01\big)\big(0\big),
\big(\circled{0}01\big)\big(\circled{0}\big),
\big(011\big)\big(0\big),
\big(011\big)\big(\circled{0}\big),
\big(\circled{0}11\big)\big(0\big),
\big(\circled{0}11\big)\big(\circled{0}\big)
\end{eqnarray*}
but only 4 hook Hecke factorizations of weight $(1,3)$:
\begin{eqnarray*}
\big(0\big)\big(100\big),
\big(0\big)\big(\circled{1}00\big),
\big(\circled{0}\big)\big(100\big),
\big(\circled{0}\big)\big(\circled{1}00\big)
\end{eqnarray*}so does not even result in a symmetric polynomial.

\end{open problem}

Finally, we end with the conjecture  we alluded to at the beginning of Section \ref{half}.   (The appendices after the references include code that can be used to check this conjecture.)

\begin{conjecture}\label{conj}
Recall the definitions of set-valued shifted tableaux and $GQ_{\lambda}(\mathbf{x})$ given in \cite{IN13}.  Then if $\gamma \supseteq \nu$ are shifted shapes we have
\begin{eqnarray*}
{{GQ}}_{\gamma/\nu}(\mathbf{x})=\sum_{\mu} b_{\gamma/\nu}^{\mu} \cdot GQ_{\mu}(\mathbf{x}),
\end{eqnarray*}
where  $b_{\gamma/\nu}^{\mu}$ is  the number of set-valued shifted  tableaux of shape $\gamma/\nu$ and weight $\mu$ having the $i$-primed property of  \ref{starting} and  $i$-lattice property of  \ref{lattice}  for all $i$.
\end{conjecture}

\begin{remark}
Note that a particular case of the conjecture implies that
\begin{eqnarray*}
{{G}}_{\lambda}(\mathbf{x},\mathbf{x})={{GQ}}_{(\lambda+\delta)/\delta}(\mathbf{x})=\sum_{\mu} b_{(\lambda+\delta)/\delta}^{\mu} \cdot GQ_{\mu}(\mathbf{x}).
\end{eqnarray*}
Since the equation \ref{first} with $\mathbf{y}=\mathbf{x}$ implies that:
\begin{eqnarray*}
{\mathcal{G}}_{w}(\mathbf{x},\mathbf{x})=\sum_{T \in HT_{w}} G_{\lambda}(\mathbf{x},\mathbf{x}),
\end{eqnarray*}
Conjecture \ref{conj} would show that ${\mathcal{G}}_{w}(\mathbf{x},\mathbf{x})$ is $GQ$-positive and give a combinatorial interpretation of the expansion coefficients.
\end{remark}

\bibliographystyle{alpha}
\bibliography{DoubleGrothendieck}

\appendix
\section{Expanding $GQ_{\gamma/\nu}$ in terms of $GQ_{\mu}$}

In this appendix we give the python code needed to determine the $GQ$ expansion of $GQ_{\gamma/\nu}$ according to Conjecture \ref{conj}.  The input syntax is $\mathtt{GQ\_expand(gamma,nu)}$ where $\mathtt{gamma}$ and $\mathtt{nu}$ are lists of integers.  The output is a list of pairs.  Each pair in the list consists of, first, a partition $\mathtt{mu}$, and, second, its multiplicity in the expansion of $GQ_{\gamma/\nu}$.
\lstset{language=Python}
\lstset{frame=lines}
\lstset{label={lst:code_direct}}
\lstset{style=mystyle}

\begin{lstlisting}






#//////////////////////////Create Sequence Function for Later//////////////////////////////
    
    

def sequences(length,maxi):
    #Create list of all sequences of fixed length using numbers {0,...,maxi}
    seq_list=[[]]
    while len(seq_list[0])<length:
        new_list=[]
        for seq in seq_list:
            for j in range(0,maxi+1):
                new_list.append(seq+[j])
        seq_list=new_list
    return(seq_list)
    
    
#/////////////////////////////One-Row Tableaux/////////////////////////////////////    
    
    
def row(m,n):
    #Function to create list of all ("shifted") one-row set-valued tableaux of length m and max entry n.
    #We use the representation 1'-->1, 1-->2, 2'-->3, 2-->4, etc.
    row_tabs=[[]]
    while len(row_tabs[0])<m:
        #(Continue until the first tab (and all other tabs) in row_tabs has the desired length.)
        #Initialize new list whose tabs will have one more box than those in row_tabs.
        new_tabs=[]        
        for rowtab in row_tabs:
            if len(rowtab)>0:
                rightmost=rowtab[-1][-1]
            else:
                rightmost=1
            #Create a list of 0-1 sequences representing all possible subsets of {rightmost,rightmost+1,...,2n}
            #0 represents in, 1 represents out.
            subset_reps=sequences(2*n+1-rightmost,1)
            for rep in subset_reps:
                #Create the subset represented by the rep
                subset=[]
                for k in range(0,len(rep)):
                    if rep[k]==0:
                        subset.append(k+rightmost)
                #Create a newtab by appending this subset to the end of rowtab:
                newtab=rowtab+[subset]
                #Check if newtab is valid
                #The new box (subset) must be nonempty and the rightmost element of previous box must either:
                #be unprimed (i.e., represented by an even integer) or
                #be primed (i.e., represented by an odd integer) and strictly less than than first element of new box or
                #not actually exist (i.e., the new box will be the first box)
                if len(subset)>0 and (rightmost%2==0 or rightmost<subset[0] or len(rowtab)==0):
                    #add subset in next box of rowtab
                    new_tabs.append(newtab)
        #Replace rowtabs with newtabs
        row_tabs=new_tabs
    return(row_tabs)    
    
        
#/////////////////////////////Check Column Requirement/////////////////////////////////////


def over (top_row,top_skew,bottom_row,bottom_skew):
    #Function to determine the validity of the two row tableau [top_row,bottom_row] where:
    #top_row is skewed to the right by top_skew units
    #bottom_row is skewed to the right by bottom_skew units 
    #bottom_row is additionally shifted one position to the right relative to top_row
    #note that top_skew>bottom_skew unless both are 0
    
    #If the bottom row extends past the top row in the rightward direction the tableau is invalid.
    if top_skew+len(top_row)<1+bottom_skew+len(bottom_row):
        return false
    
    #Check the column condition for each box in top_row.
    for i in range(0,len(top_row)):
        top_box=top_row[i]
        #Determine if there is a box below top_box and if so what it contains
        j=(top_skew+i)-(bottom_skew+1)
        if 0<=j<len(bottom_row):
            #Then the box below exists
            bottom_box=bottom_row[j]
            if max(top_box)>min(bottom_box):
                #Larger number over smaller number --> Column condition broken
                return False
            if max(top_box)==min(bottom_box) and max(top_box)%2==0:
                #An unprimed number over the same unprimed number --> Column condition broken
                return False
    return True   
   


#/////////////////////////////Create Candidates for Stembridge Tableaux/////////////////////////////////////

def row_flagged(gamma,nu):
    #We start by noting that any tableau with the i-lattice property for all i cannot have a j>i in row i.
    #So first return all shifted set-valued tableau of shape lam/mu such that entries in row i are <=i.
    
    #Step 1: Ensure the partitions have same length
    while len(nu)<len(gamma):
        nu.append(0)
        
    flag_tabs=[]
    #Step 2: Initialize flag_tabs to the set of one-row tableau in the alphabet {1',1}
    base_rows=row(gamma[0]-nu[0],1)
    for baserow in base_rows:
        flag_tabs.append([baserow])
        #base_rows is a list of lists (rows): flag_tabs is a list of LISTS (tableaux--with one row right now) of lists
        
    #Step 3: Redefine flag_tabs until it is a list of tableaux each of which has |gamma| rows.   
    for j in range(1,len(gamma)):
        new_tabs=[]
        new_rows=row(gamma[j]-nu[j],j+1)
        for tab in flag_tabs:
            last_row=tab[-1]
            for newrow in new_rows:
                #Check if adding this newrow is valid
                if over(last_row,nu[j-1],newrow,nu[j])==True:
                    new_tabs.append(tab+[newrow])
        flag_tabs=new_tabs
    return(flag_tabs)



#////////////////////////////////////Read the word of a tableau//////////////////////////////
            
            
def read(P):
    #Return reading word: bottom row to top row; left to right within rows;  preserving whatever intra box order
    r=[]
    for i in range(len(P)-1,-1,-1):
        for j in range(0,len(P[i])):
            for k in range(0,len(P[i][j])):
                r.append(P[i][j][k])
    return(r)





#/////////////////////////////////Sort as in Definition 5.15 and 5.16////////////////////////
    
def sort_boxes(P):
    #Function to sort the elements within a box in accordance with Definition 5.15 and 5.16
    #Returns new tableau with the elements in the boxes sorted in the order 1,2,3,...,3',2',1'
      
    Q=[]
    for row in P:
        Q.append([])
        for box in row:
            newbox=copy.copy(box)
            newbox.sort(key=lambda x: -1/(x*math.pow(-1,x%2)))
            #Recall boxes actually contain letters from {1,2,3,4,...} corresponding to {1',1,2',2,...} 
            Q[-1].append(newbox)
    return(Q)


#////////////////////////////////////Check Definition 5.15//////////////////////////////


def primed_prop(P):
    Q=sort_boxes(P)
    #Function to check Definition 5.15 (the i-primed property)
    word=read(Q)
    #Recall word is composed of letters from {1,2,3,4,...} corresponding to {1',1,2',2,...}
    maxi=max(word)

    starts=[0]*(math.ceil(maxi/2)+1)
    #(This vector counts whether an i or i' has appeared (starts[0] always is 0 by convention))
        
    #We simply need to read through the word left to right due to how it was constructed.
    for letter in word:
        #Recall letter is the representative from {1,2,3,4,...} corresponding to {1',1,2',2,...}
        base=math.ceil(letter/2)
        #Now, we check whether this is the first i or i' appearing and if so record what type it is
        if starts[base]==0:
            #0 indicates no i or i' has appeared yet
            if letter%2==0:
                starts[base]=1
                #record that the first i or i' is an i with a 1
            if letter%2==1:
                starts[base]=-1
                #record that the first i or i' is an i' with a -1
    ok=True
    for k in range(0,len(starts)):
        if starts[k]==1:
            ok=False
    return(ok)

#////////////////////////////////////Check Definition 5.16//////////////////////////////

def lattice_prop(P):
    #Does the tableau P have the lattice property?
    Q=sort_boxes(P)
    word=read(Q)
    maxi=max(word)
    counts=[0]*(math.ceil(maxi/2)+2)
    #This keeps track of the number of times an i appears when reading backwards and an i' appears when reading forward.
    #By convention counts[0]=0 and counts[-1]=0 always.

    #read backwards
    for j in range(len(word)-1,-1,-1):
        letter=word[j]
        base=math.ceil(letter/2)
        #while reading backwards each time you read an i add 1 to counts[i] 
        if letter%2==0:
            counts[base]+=1
            #if counts[i] becomes greater than counts[i-1] return false
            if base>1 and counts[base]>counts[base-1]:
                return False
        #each time you read an i' while reading backward:
        if letter%2==1:
            #counts does not change
            #if you read an i' while counts[i] equals counts[i-1] return false
            if base>1 and counts[base]==counts[base-1]:
                return False  
    #read forwards starting with the currect value of counts
    for j in range(0,len(word)):
        letter=word[j]
        base=math.ceil(letter/2)
        #while reading forward add 1 to counts[i] each time you read an i'
        if letter%2==1:
            counts[base]+=1
            #if counts[i] becomes greater than counts[i-1] return false
            if base>1 and counts[base]>counts[base-1]:
                return False
        #each time you read an i while reading forward:
        if letter%2==0:
            #counts does not change
            #if you read an i while counts[i+1] equals counts[i] return false
            if base>0 and counts[base+1]==counts[base]:
                return False          
    return(True)


#////////////////////////////////////Find the Stembridge tableaux//////////////////////////////



def stembridge(gamma,nu):
    S=[]
    Candidates=row_flagged(gamma,nu)
    for P in Candidates:
        if primed_prop(P)==True and lattice_prop(P)==True:
            S.append(P)
    return(S)
                


#/////////////////////////////////////Compute the Weight//////////////////////////////
 
def weight(word):
    #Recall word is composed of letters from {1,2,3,4,...} corresponding to {1',1,2',2,...}
    maxi=max(word)
    counts=[0]*(math.ceil(maxi/2))
    #This vector will count the number of times an i or i' appears.
    for letter in word:
        base=math.ceil(letter/2)
        counts[base-1]+=1
        #base-1 because our tableaux do not have 0s
    return(counts)

                        
#/////////////////////////////////////Find the GQ expansion//////////////////////////////
   
def GQ_expand(gamma,nu):
    #Write GQ_{gamma/nu} as a sum of GQ_{mu} as in Conjecture 5.1
    
    #Find all the gamma/nu Stembridge tableaux
    S=stembridge(gamma,nu)
    #Create a list of their weights
    W=[]
    for tab in S:
        W.append(weight(read(tab)))     
    #Sorts the weights
    W.sort(key=lambda x: str(x)[1:-1])
    W.reverse()
    #Record the weights along with the multiplicity with which they appear
    U=[[W[0],1]]
    for i in range(1,len(W)):
        if W[i]==W[i-1]:
            U[-1][1]+=1
        else:
            U.append([W[i],1])
    #Return a list of pairs: First element is a partition, second element is its GQ multiplicity in GQ_{gamma/nu}. 
    return(U)


GQ_expand([5,4,2],[3,1])




\end{lstlisting}

\section{Checking Conjecture \ref{conj} for fixed $\gamma/\nu$}

Recall $GQ_{\gamma/\nu}(\mathbf{x})$ is the generating function over set-valued shifted  tableaux as they are defined in \cite{IN13}. Now let $GR_{\gamma/\nu}(\mathbf{x})$ be the generating function over the  set-valued shifted  tableaux of \cite{IN13} which happen to have the $i$-primed property  of definition \ref{starting} for all $i$. Note that each set-valued shifted tableau, $T$, with the $i$-primed property for all $i$ can be seen as giving rise to $3^{a}$ set-valued shifted tableaux, where $a$ is the number of values of $i$ for which at least one $i'$ or $i$ shows up in $T$.    These tableaux arise from $T$ by, for each such $i$, replacing the first $i'$ in the tableau with one of  $i'$, $i$, or $(i',i)$. Keeping this in mind it is not hard to see that if we know $GR_{\gamma/\nu}(\mathbf{x})$ we can compute $GQ_{\gamma/\nu}(\mathbf{x})$ (each degree $d$ monomial in the former will give rise to $3^{a}$ monomials in the latter ranging in degree from $d$ to $d+a$ where $a$ is the number of variables in the original monomial with non-zero exponent).  In fact, it is not difficult to show that Conjecture \ref{conj} follows from showing that
\begin{eqnarray*}
{{GR}}_{\gamma/\nu}(\mathbf{x})=\sum_{\mu} b_{\gamma/\nu}^{\mu} \cdot GR_{\mu}(\mathbf{x}).
\end{eqnarray*}
Finally,  the fact that $GQ_{\gamma/\nu}(\mathbf{x})$ is symmetric implies that $GR_{\gamma/\nu}(\mathbf{x})$ is symmetric so  it suffices to check that the same dominant monomials appear (with the same mulitplicities) on the left and right side of the equation above to establish Conjecture \ref{conj}. This check is performed by the code below for fixed parameters. The input syntax is $\mathtt{Test\_Conjecture(gamma,nu,num\_vars,up\_to\_degree}$.  This performs the check for  shape $\gamma/\nu$ for all monomials in the finite set of variables $(x_1,\ldots,x_{num\_vars})$ that have degree at most $\mathtt{up\_to\_degree}$.    The code in the previous appendix is needed to run this code.

\lstset{language=Python}
\lstset{frame=lines}
\lstset{label={lst:code_direct}}
\lstset{style=mystyle}

\begin{lstlisting}
import copy
import math

#///////////Function to determine all standard shifted tableaux of shape gamma/nu in the letters {0,...,alph-1}

def standard_tabs(gamma,nu,alph):  
    #make nu and gamma have same length
    while len(nu)<len(gamma):
        nu.append(0)
    
    #Create an empty tableau with len(gamma) rows. 
    #Add ["x"] in each location of a box that is skewed or shifted out.
    empty_tab=[]
    for i in  range(0,len(gamma)):
        offset=nu[i]+i+1
        empty_row=[]
        for j in range(0,offset):
            empty_row.append(["x"])
        empty_tab.append(empty_row)
 
    #Initialize a list with a pair composed of this tableau and an empty positions set
    positions=[]
    tab_list=[    [     empty_tab,    positions   ]     ]
    #positions is the list of the (x,y)-coordinates of the entries of the tableau
    #positions will be used to determine the peak_set and repeat_set of the tableau
    
    #Create all tableaux by successively adding each n from 0 to alph-1 to a new box or a terminal box
    for n in range(1,alph+1):
        new_list=[]
        for old in tab_list:
            old_tab=old[0]
            positions=old[1]            
            for i in range(0,len(old_tab)):
                #try to add n to last box of row i 
                if old_tab[i][-1]!=["x"]:
                    if i==len(old_tab)-1 or old_tab[i+1][-1]==["x"] or len(old_tab[i])>len(old_tab[i+1]):
                        new_tab=copy.deepcopy(old_tab) 
                        new_tab[i][-1]+=[n]
                        new_position=[i,len(new_tab[i])]
                        new_list.append([new_tab,positions+[new_position]])
                #also try to add new box to row i with [n]
                if len(old_tab[i])<gamma[i]+i+1:
                    if i==0 or len(old_tab[i])<len(old_tab[i-1]):
                        new_tab=copy.deepcopy(old_tab) 
                        new_tab[i]+=[[n]]
                        new_position=[i,len(new_tab[i])]
                        new_list.append([new_tab,positions+[new_position]])
                    

        tab_list=new_list
    
    
    #Create the final list of tableaux and record their peak and repeat sets.
    P=[]
    for tab in tab_list:
        #Check that each box of gamma/nu has been filled with at least one entry. 
        boxes_full=True
        for h in range(len(gamma)):
            if len(tab[0][h])!=gamma[h]+h+1:
                boxes_full=False

        if boxes_full==True:
            #compute the peak_set and repeat_set for each tableau
            #the peak_set is the set of j such that j-1 lies strictly left of j and j+1 lies strictly below j
            #the repeat_set is the set of i such that i+1 lies in the same box as i.
            peak_set=[]
            repeat_set=[]
            positions=tab[1]
            for i in range(0,len(positions)-2):
                j=i+1
                k=i+2
                if positions[j][1]>positions[i][1] and positions[k][0]>positions[j][0]:
                    peak_set.append(j)
            for i in range(0,len(positions)-1):
                if positions[i]==positions[i+1]:
                    repeat_set.append(i)
            P.append([tab[0],peak_set,repeat_set])
                    
    return(P)



#////////////Function to return list of all partitions of n into at most k parts//////////
def partitions(n,k):
    par_list=[[1]]
    while sum(par_list[0])<n:
        new_list=[]
        #Any partition of size m+1 and at most k rows can be uniquely created by
        #Starting with a partition of size m, say old_par
        for old_par in par_list:
            #And then either:
            #Trying to add a box to the last row of old_par
            if len(old_par)==1 or (len(old_par)>1 and old_par[-1]<old_par[-2]):
                new_par=copy.copy(old_par)
                new_par[-1]+=1
                new_list.append(new_par)
            #Or trying to add a new row with one box to old_par
            if len(old_par)<k:
                new_par=copy.copy(old_par)
                new_par.append(1)
                new_list.append(new_par)
        par_list=new_list
    return(par_list)


#///////Returns all (dominant) weights that arise from a valid semi-standardization of a given standard_tab
#(Set-valued shifted tableaux (with i-primed prop) are in bijection with pairs (standard_tab, valid weight))
#See the definiction of the function for a precise understanding of 'valid'
def polynomial(standard_tab,num_vars):
    #We are given a standard tab along with peak and repeat sets:
    tab=standard_tab[0]
    peak_set=standard_tab[1]
    repeat_set=standard_tab[2]
    #Compute number of entries in tab so we know what size of partitions to check.
    num_entries=0
    for row in tab:
        for box in row:
            if box!=["x"]:
                num_entries+=len(box)
    #Find all partitions of num_entries into at most num_vars parts 
    pars=partitions(num_entries,num_vars)
    #For each partition we check if there is a semi-standarization of tab with this weight.
    Weights=[]
    for par in pars:
        #Express partition as a  weakly increasing sequence with par[j] entries equal to j+1 
        weak_seq=[]
        for j in range(0,len(par)):
            weak_seq+=[j+1]*par[j]

        #Determine whether this sequence is a valid semi-standardization of tab
        good=True
        #Check that whenever weak_seq[k]=weak_seq[k+1], k and k+1 are in different boxes
        for k in range(0,len(weak_seq)-1):
            if weak_seq[k]==weak_seq[k+1] and k in repeat_set:
                good=False
        #Check that whenever weak_seq[j]=weak_seq[j+1]=weak_seq[j+2], j+1 is not a peak
        for j in range(0,len(weak_seq)-2):
            if weak_seq[j]==weak_seq[j+1] and weak_seq[j+1]==weak_seq[j+2] and j+1 in peak_set:
                good=False
        #This ensures that the set of consecutive entries of tab corresponding to a fixed value 
        #under semi-standardization form a 'vee' in the tableau
        if good==True:
            Weights.append(par)
    return(Weights)



    
#//////////Find the monomial expansion of GR_{gamma/mu} in num_vars variables of degree at most up_to_degree
def monomial_exp(gamma,nu,num_vars,up_to_degree):
    #Compute the (dominant) monomials appearing in GR_{gamma/nu}.
    P=[]
    alph=sum(gamma)-sum(nu)
    while alph<=up_to_degree:
        print("-computing terms of degree "+str(alph)+".") 
        #Q will hold the monomials of degree alph.
        Q=[]
        #Find the standard tabs of shape gamma/mu using alph numbers
        tabs=standard_tabs(gamma,nu,alph)
        for tab in tabs:
            #Find all the (semi-standard) set-valued shifted tableau (with i-primed prop) and dominant weight
            #that are asociated to tab. And add the terms corresponding to their weights to Q
            Q+=polynomial(tab,num_vars)
        if Q!=[]:
            #Add Q to P
            P+=Q
        else:
            alph=up_to_degree
        alph+=1
    return(P)


#Compare GR_{gamma/nu} calculated directly with GR_{gamma/mu} expanded into GR_{mu}'s and then into monomials

def Test_Conjecture(gamma,nu,num_vars,up_to_degree):
    print("Computing the value of GR_{"+str(gamma)+"/"+str(nu)+"} directly")
    direct=monomial_exp(gamma,nu,num_vars,up_to_degree)
    direct.sort(key=lambda x: str(x))
    print("There are "+str(len(direct))+" dominant monomials")
    
    
    conjectured=[]
    print("Computing the conjectured expansion.")
    Expansion=GQ_expand(gamma,nu)
    print("There are "+str(len(Expansion))+" distinct shapes in the expansion.")
    for pair in Expansion:
        print("Computing the value of GR_{"+str(pair[0])+"}")
        shape=pair[0]
        shape_mult=pair[1]
        monomials=monomial_exp(shape,[],num_vars,up_to_degree)
        #Add the monomials from this shape to the list as many times as the shape shows up.
        for i in range(0,shape_mult):
            conjectured+=monomials
    conjectured.sort(key=lambda x: str(x))



    return(direct==conjectured)


        
Test_Conjecture([5,4,2],[3,1],3,9)   
\end{lstlisting}

\end{document}